\documentclass[12pt]{article}
 
\usepackage[margin = 1in]{geometry}
\usepackage{amsfonts,latexsym,amsmath,amssymb,amsthm}
\usepackage[mathscr]{eucal}
\usepackage{graphicx,color}
\usepackage{comment}
\usepackage{hyperref}
\usepackage{epstopdf}

\usepackage{mathrsfs} 

\newtheorem{theorem}{Theorem}
\newtheorem{lemma}{Lemma}

\newtheorem{definition}{Definition}

\newtheorem{remark}{Remark}

\newcommand{\N}{\ensuremath{\mathbb{N}}}
\newcommand{\R}{\ensuremath{\mathbb{R}}}

\newcommand{\mom}{\mathbf{z}}
\newcommand{\mmon}{\mathbf{v}}
\newcommand{\rf}{\ell}

\newcommand{\optmeas}{\nu}
\renewcommand{\d}{\,\mathrm{d}}
\newcommand{\meas}{\mathscr{M}} 		
\newcommand{\prob}{\mathscr{P}} 		
\newcommand{\C}{\mathscr{C}} 
\newcommand{\Leb}{\mathscr{L}} 

\newcommand{\spt}[1]{\mathrm{spt}(#1)}

\newcommand{\genvar}{\mathbf{w}}

\newcommand{\rx}{\R[\genvar]}

\newcommand{\K}{\mathbf{K}}

\newcommand{\p}{\mathrm{\mathbf{p}}}
\newcommand{\T}{\mathbf{T}}
\newcommand{\X}{\mathbf{X}}
\newcommand{\Y}{\mathbf{Y}}
\newcommand{\yL}{\underline{y}}
\newcommand{\yR}{\bar{y}}
\def\KLT{\mathbf{K}_{T}}
\def\KLO{\mathbf{K}_{0}}
\def\KLL{\mathbf{K}_{L}}
\def\KLR{\mathbf{K}_{R}}

\def\KL{\mathbf{K}}

\def\ELTp{\mathbf{K}^+_{T}}
\def\ELTm{\mathbf{K}^-_{T}}
\def\ELOp{\mathbf{K}^+_{0}}
\def\ELOm{\mathbf{K}^-_{0}}

\def\measO{\sigma_0}
\def\measT{\sigma_T}
\def\measL{\sigma_L}
\def\measR{\sigma_R}

\def\ELLp{\mathbf{K}^+_{L}}
\def\ELLm{\mathbf{K}^-_{L}}

\def\ELRp{\mathbf{K}^+_{R}}
\def\ELRm{\mathbf{K}^-_{R}}

\def\ELp{\mathbf{K}^+}
\def\ELm{\mathbf{K}^-}

\def\sign{\mathrm{sign}}

\usepackage[normalem]{ulem}
\usepackage{color}

\begin{document}

\title{ \LARGE \bf%
	A moment approach for entropy solutions to nonlinear hyperbolic PDEs\footnote{This work was partly funded by the ERC Advanced Grant Taming and by project 16-19526S of the Grant Agency of the Czech Republic.}
}
\author{Swann Marx$^1$, Tillmann Weisser$^1$, Didier Henrion$^{1,2}$ and Jean Bernard Lasserre$^{1,3}$}

\footnotetext[1]{CNRS; LAAS; Universit\'e de Toulouse; 7 avenue du colonel Roche, F-31400 Toulouse; France}
\footnotetext[2]{Faculty of Electrical Engineering, Czech Technical University in Prague, Technick\'a 4, CZ-16206 Prague, Czechia}
\footnotetext[3]{IMT; UPS; Universit\'e de Toulouse; 7 avenue du colonel Roche, F-31400 Toulouse; France}

\maketitle

\abstract{We propose to solve polynomial hyperbolic partial differential equations (PDEs) with convex optimization. This approach is based on a very weak notion of solution of the nonlinear equation, namely the measure-valued (mv) solution, satisfying a linear equation in the space of Borel measures. The aim of this paper is, first, to provide the conditions that ensure the equivalence between the two formulations and, second, to introduce a method which approximates the infinite-dimensional linear problem by a hierarchy of convex, finite-dimensional, semidefinite programming problems. This result is then illustrated on the celebrated Burgers equation. We also compare our results with an existing numerical scheme, namely the Godunov scheme.\\[1em]
	}

\section{Introduction}

This paper is concerned with the numerical study of scalar {\em nonlinear hyperbolic conservation laws},
a partial differential equation (PDE) which model numerous physical phenomena such as fluid mechanics, traffic flow or nonlinear acoustics \cite{dafermos2006hyperbolic}, \cite{whitham2011linear}. The existence and uniqueness of solutions to the associated Cauchy problem crucially depends on the flux and the initial condition \cite{kruvzkov1970first}. Even if the solution is unique, its numerical computation is still a challenge -- in particular when the solution has a shock, i.e., a discontinuity. Existing schemes based on discretization such as \cite{godunov1959difference} suffer from numerical dissipation: the shock is smoothened in the numerical solution and cannot be represented accurately. 
In fact, sometimes the exact location of the shock is of crucial interest for applications. Note however that some existing numerical schemes are able to capture shocks in the case where the conservation laws under consideration are linear, see e.g. \cite{Despres2001}.

In contrast to existing methods, a distinguishing feature of the numerical scheme presented in this paper is to {\em not} rely on discretization; it computes the solution in a given time-space window globally. From such a solution, the location of the shock at a given time can be computed up to the limits of machine precision. In our opinion this is a major advantage when compared to other numerical methods. 

\textbf{Measure-Valued Solutions} While PDEs are usually understood in a weak sense, DiPerna proposed an even weaker notion of solution, so-called {\em measure-valued solutions} (mv solutions for short) \cite{diperna1985measure}, which are based on Young measures, i.e. time and/or space dependent probability measures. Young measures have originally been introduced in the context of calculus of variations and optimal control, where 
the velocity or more generally the control is relaxed from being a function of time to being a time-dependent probability measure on the control space, see e.g. \cite[Part III]{fattorini1999infinite} for an overview.  Similarly, DiPerna introduced mv solutions to conservation laws as measures on the solution space, now depending on time and space. 

Naturally, every weak solution gives rise to a mv solution when identifying a solution $y(t,x)$ with the Young measure $\delta_{y(t,x)}(dy)$. We say then that the mv solution is {\it concentrated} on (the graph of) the solution. In this paper, we are focusing on a setup where both weak and mv solutions are unique (hence identical). In this case, both solutions coincide via the identification just mentioned. Note however that our approach also applies without any change to situations where the mv solution is not concentrated, e.g., because of an initial condition that is not concentrated either.

In order to ensure uniqueness we rely on the notion of {\em entropy solutions} which has been extended to entropy mv solutions. Entropy is a concept from thermodynamics that makes reference to the fact that differences in physical systems, e.g., the densities of particles in a room, tend to adjust to each other. It is well-known that the entropy solution of a scalar nonlinear hyperbolic conservation law is unique. For the generalized situation things are more involved. However under 
suitable assumptions on the initial condition entropy, uniqueness of mv solutions can be proved.

Recently there has been an increasing interest in numerical schemes to compute mv solutions for hyperbolic conservation laws with non concentrated initial conditions \cite{fjordholm2017construction,feireisl2018convergence}. 
Existing numerical schemes apply standard discretization methods to compute sufficiently many trajectories according to the distribution of the initial condition and recover the moments of the mv solution by considering limits of the trajectories. In contrast to this our approach directly computes the moments of the mv solution. Therefore in some sense this work is in the opposite direction. We compute moments to recover trajectories in the case where the initial condition and the solution are concentrated.

\textbf{Generalized Moment Problem} The key idea underlying the approach is to consider mv solutions as solutions to a particular instance of the {\em Generalized Moment Problem} (GMP) which is an infinite-dimensional optimization problem on appropriate spaces of measures, and where both the cost and the (possibly countably many) constraints are {\em linear} in the moments of the respective measures. Lasserre \cite{lasserre2009moments} showed that the GMP can be approximated as closely as desired by solving a hierarchy of convex {\em semidefinite programs} (SDP) of increasing size, provided that the data of the GMP are semi-algebraic; that is, the measures are supported on basic semi algebraic compact sets (i.e. bounded sets defined by finitely many polynomial inequalities and equations) and the involved functions are polynomial or semi-algebraic functions  (i.e. functions whose graphs are semi-algebraic sets). The duals to these SDPs are linear problems on polynomial sums of squares (SOS). Therefore this hierarchy of SDP relaxations is called the {\em moment-SOS (sums of squares) hierarchy}. By now, many problems from different fields of mathematics, including optimal control of ordinary differential equations \cite{lasserre2008nonlinear}, have been reformulated as particular instances of the GMP and then approximated or solved by the moment-SOS hierarchy. This paper is in the line of
these former contributions. That is, (i) the mv solutions are viewed (or formulated) as solutions
of a particular instance of the GMP, and (ii) the moments of mv solutions are 
approximated as closely as desired by solving a moment-SOS hierarchy.

Any optimal solution of each semidefinite relaxation at step $d$ in the hierarchy provides information about the mv solution in the form of a sequence of its (approximated) moments, up to degree $d$; the higher is $d$ the better is the approximation of its moments. As we restrict to measures with compact support,  they are fully characterized from knowledge of the complete sequence of their moments. Interestingly, it is worth noting that in \cite{fjordholm2017construction} it was already pointed out that the statistical moments of mv solutions are precisely the quantities of interest. 

\textbf{Contribution} 
To the best of our knowledge, this work seems to be the first
contribution where nonlinear PDEs are addressed {\em without time-space domain discretization}
and using convex optimization {\em with a  proof of convergence}. An original early attempt to compute mv solutions of nonlinear wave equations with linear programming was reported in \cite{rubio1997}, also in the presence of controls. In \cite{lasserre2008nonlinear}, the authors apply to moment-SOS hierarchy to solve optimal control problems of ordinary differential equations, and it was shown in \cite{pauwels2017} that it provides a sequence of subsolutions converging in norm to the viscosity solution of the Hamilton-Jacobi-Bellmann PDE, a particular nonlinear hyperbolic equation. In \cite{mevissen2011}, nonlinear PDEs are discretized into large-scale sparse polynomial optimization problems, in turn solved with the moment-SOS hierarchy. More recently, bounds on functionals of solutions were obtained with SOS polynomials for nonlinear PDEs arising in fluid dynamics in \cite{chernyshenko2014polynomial} and for the nonlinear Kuramoto-Sivashinsky PDE in \cite{goluskin2018bounds}. These works, however, focus only on the dual SOS problems, and they provide bounds with no convergence guarantees. They do not exploit the primal formulation of the problem on moments, which we believe to be crucial for convergence analysis.
In the recent work \cite{mishra2016}, the authors compute mv solutions for the equations of compressible and incompressible inviscid fluid dynamics, with the help of discretization algorithms based on Monte Carlo methods. Even more recently, in \cite{brenier2017optimization} the author has proposed a convex formulation for the classical solution to nonlinear hyperbolic PDEs and he proves that the entropy solution to the Burgers equation might be recovered also via this optimization problem. However, this paper does not provide a numerical scheme. In the concurrent work \cite{korda2018}, the authors propose to use the moment-SOS hierarchy in a much more general setting of a controlled polynomial PDE. However, at that level of generality, there is no proof that the numerical scheme will converge to an appropriate solution of the PDE. For more references on previous attempts to use convex optimization for solving and controlling PDEs, the reader is referred to the introduction of \cite{korda2018}.

\textbf{Outline} This paper is organized as follows. Section \ref{1.solution} introduces different notions of solutions for scalar conservation laws and provides some links between these notions. Section \ref{2.convex} introduces the Moment-SOS hierarchy, proves that the mv solution framework can be written as an instance of the GMP, and explains how one may interpret the moment solutions. The focus of Section \ref{3.riemann} is  on a numerical study of the Burgers equation. Finally, Section \ref{5.conclusion} collects some concluding remarks and further research issues to be addressed.

\textbf{Notation} If $\mathcal{X}$ is a topological space, let $\C(\mathcal{X})$ resp. $\C_0(\mathcal{X})$  resp. $\C^1_c(\mathcal{X})$ denote the space of functions on $\mathcal X$ that are continuous resp. continuous and vanishing at infinity resp. continuously differentiable with compact support. For $p\geq 1$, the Lebesgue space $\Leb^p(\mathcal{X})$ consists of functions on $\mathcal{X}$ whose $p$-norms are bounded. The set of signed resp. positive Borel measures is denoted $\meas(\mathcal{X})$ resp. $\meas(\mathcal{X})_+$. The set of probability measures on $\mathcal{X}$ is denoted by $\prob(\mathcal{X})$ and it consists of elements $\mu \in \meas(\mathcal{X})_+$ such that $\mu(\mathcal{X})=1$. The measure $\lambda_{\mathcal{X}} \in \prob(\mathcal{X})$ denotes the normalized Lebesgue measure on $\mathcal{X}$. Given a vector $\genvar=(w_1 \: \ldots \: w_n)$, we denote by $\rx$ the ring of real multivariate polynomials in the variables $w_1,\ldots,w_n$.

\section{Notions of solutions}
\label{1.solution}

We start with a brief overview of different notions of solutions to scalar polynomial PDEs. For details, we refer to \cite{dafermos2006hyperbolic} for weak solutions and \cite{necas1996weak} for measure-valued solutions. The aim of this section is to give a clear link between these two concepts of solutions. 

\subsection{Weak and entropy solutions}
In order to study mv solutions, it is instructive to revisit the classical concept of weak solutions first. Consider therefore the Cauchy problem 
\begin{subequations}
\label{eq:burgers}
\begin{align}
&\frac{\partial y}{\partial t}(t,x) + \frac{\partial f(y)}{\partial x}(t,x) =0, \quad (t,x)\in\mathbb{R}_+\times \mathbb{R}, \label{eq:burgers: conversation law}\\
&y(0,x)=y_0(x),\quad x\in\mathbb{R},\label{eq:burgers: initial condition}
\end{align}
\end{subequations}
where \eqref{eq:burgers: conversation law} is a scalar  hyperbolic conservation law with $f\in \C^1(\R)$ and \eqref{eq:burgers: initial condition} provides an initial condition $y_0\in \Leb^1(\mathbb{R})\cap \Leb^\infty(\mathbb{R})$. Note that \eqref{eq:burgers: conversation law} encompasses, among others, the well-known Burgers equation if one sets $f(y)=\frac{1}{2}y^2$. 

Even if the initial condition $y_0$ is smooth, solutions to \eqref{eq:burgers} might be discontinuous (see \cite[p. 143]{evans2010pde} for the case of the Burgers equation). Solutions to this problem are hence usually understood in the following weak sense.
\begin{definition}[Weak solution]
A function $y\in \Leb^\infty(\mathbb{R}_+\times \mathbb{R})$ is a weak solution to \eqref{eq:burgers} if, for all test functions $\psi_1\in \C^1_c(\mathbb{R}_+\times \mathbb{R})$, it satisfies
\begin{equation}\label{y}
\int_{\mathbb{R}_+}\int_{\mathbb{R}} \left(\frac{\partial \psi_1(t,x)}{\partial t} y(t,x) + \frac{\partial \psi_1(t,x)}{\partial x} f(y(t,x)) \right)\, dx\,dt +  \int_{\mathbb{R}} \psi_1(0,x) y_0(x)\, dx=0.
\end{equation}
\end{definition}

In general, weak solutions to \eqref{eq:burgers} are not unique. However it can be shown (see e.g. \cite{kruvzkov1970first}) that among all possible weak solutions, only one has a physical meaning. This solution is called the \emph{entropy solution} and can be characterized as follows.
\begin{definition}[Entropy pair/entropy solution]
\begin{itemize}
\item[(i)]A pair of functions $\eta,q\in \C^1(\mathbb{R})$ is called an \emph{entropy pair} for \eqref{eq:burgers: conversation law} if $\eta$ is strictly convex and $q^\prime= f^\prime\eta^\prime$.\\
\item[(ii)] A weak solution $y\in \Leb^\infty(\mathbb{R}_+\times\mathbb{R})$ of \eqref{eq:burgers} is an \emph{entropy solution} if, for all entropy pairs and all non-negative test functions $\psi_2\in \C^1_c(\mathbb{R}_+\times \mathbb{R})$, it satisfies
\begin{equation}
\label{entropy-sol}
\int_{\mathbb{R}_+}\int_{\mathbb{R}} \left(\frac{\partial\psi_2}{\partial t}\eta(y) + \frac{\partial \psi_2}{\partial x}q(y)\right) dx\, dt + \int_{\mathbb{R}} \psi_2(0,x) \eta(y_0)(x) \,dx\geq 0.
\end{equation}
\end{itemize}
\end{definition}

\subsection{Measure-valued solutions}

Generally, regularity results of conservation laws are obtained from regularized conservation laws
\begin{equation*}
\frac{\partial}{\partial t} y^\varepsilon + \frac{\partial}{\partial x}f(y^\varepsilon) - \varepsilon \frac{\partial^2}{\partial x^2} y^\varepsilon=0,\
\end{equation*}
where $\varepsilon>0$ is a fixed parameter.
Then one studies the limit of solutions $y^\varepsilon$ as $\varepsilon$ goes to $0$ and tries to retrieve some of regularity properties of the latter equation for the conservation law. However, on the one hand, regularized solution $y^\varepsilon$ may or may not converge to a weak solution $y$ of \eqref{eq:burgers}. This is due to a lack of reflexivity of the space $\Leb^\infty$. On the other hand, regularized solutions $y^\varepsilon$ necessarily converge to a measure-valued (mv) solution. This notion builds upon the concept of a Young measure.

\begin{definition}[Young measure]
A Young measure on a Euclidean space $\mathcal{X}$ is a map $\mu : \mathcal{X} \rightarrow \prob(\mathbb{R})$, $\xi \mapsto \mu_{\xi}$, such that for all $g\in \C_0(\mathbb{R})$ the function $\xi \mapsto \int_\R g(y)\mu_{\xi}(dy)$ is measurable.
\end{definition}

Later, mv solutions have also proved to be useful in the study of problems more general than \eqref{eq:burgers}, where the initial condition \eqref{eq:burgers: initial condition} is replaced by a Young measure parametrized in space (see e.g. \cite{fjordholm2017construction} and the references therein). The generalized problem is to find a Young measure $\mu_{(t,x)}$ which satisfies the following Cauchy problem:
\begin{subequations}
\label{eq:burgers mv}
\begin{align}
&{\partial_t}\langle \mu_{(t,x)}, y \rangle + {\partial_x}\langle\mu_{(t,x)} ,f(y)\rangle = 0, \quad (t,x)\in\mathbb{R}_+\times \mathbb{R}, \label{eq:burgers mv: conservation law}\\
&\mu_{(0,x)}=\measO,\quad x\in\mathbb{R},\label{eq:burgers mv: initial condition}
\end{align}
\end{subequations}
where $\langle\cdot, \cdot \rangle$ denotes integration of a measure $\mu \in \meas(\mathbb{R})$ against a function $g\in \C(\mathbb{R})$:
\begin{equation*}
\langle \mu,g \rangle := \int_{\mathbb{R}} g(y) \mu(dy).
\end{equation*} 
In \eqref{eq:burgers mv} the measure $\measO$ is a given Young measure on $\R$, and $f$ is a continuously differentiable function on $\R$. The conservation law \eqref{eq:burgers mv: conservation law} has to be understood in the sense of distributions, i.e.:
\begin{definition}[Measure-valued solution]
A Young measure is a measure-valued (mv) solution to \eqref{eq:burgers mv} if, for all test functions $\psi_1\in \C^1_c(\mathbb{R}_+\times\mathbb{R})$, it satisfies 
\begin{equation}\label{eq:mv sol}
\int_{\mathbb{R}_+}\int_{\mathbb{R}} \left( \frac{\partial \psi_1(t,x)}{\partial t} \langle \mu_{(t,x)},y\rangle + \frac{\partial \psi_1(t,x)}{\partial x} \langle \mu_{(t,x)},f(y)\rangle\right) dx\,dt + \int_{\mathbb{R}} \psi_1(0,x)\langle \measO,y\rangle dx = 0.
\end{equation}
\end{definition}
Note that the weak solution $y$ has been replaced by a time-space parametrized probability measure $\mu$ supported on the range of $y$. Whereas a weak solution is requested to satisfy \eqref{y}, only averages of the mv solution are considered in \eqref{eq:mv sol}. 
It is easy to see that every weak solution induces a mv solution via the canonical embedding $y(t,x)\mapsto\delta_{y(t,x)}$. As in the case of weak solution, an entropy condition is needed in order to select solutions with a physical meaning. Quite in analogy to entropy solutions, entropy mv solutions are defined as follows.

\begin{definition}[Entropy measure-valued solution]
An mv solution $\mu$ is an entropy mv solution to \eqref{eq:burgers mv} if, for all entropy pairs $(\eta,q)$ and all non-negative test functions $\psi_2\in \C^1_c(\mathbb{R}_+\times\mathbb{R})$, it satisfies
\begin{equation}
\label{mv-entropy}
\int_{\mathbb{R}_+}\int_{\mathbb{R}} \left( \frac{\partial \psi_2(t,x)}{\partial t} \langle \mu_{(t,x)},\eta\rangle + \frac{\partial \psi_2(t,x)}{\partial x} \langle \mu_{(t,x)},q\rangle\right) dx\,dt + \int_{\mathbb{R}} \psi_2(0,x)\langle \measO,\eta\rangle dx \geq 0.
\end{equation}
\end{definition}
\begin{remark} Again it is straightforward to see that entropy solutions are entropy mv solutions via the canonical embedding $y(t,x)\mapsto \delta_{y(t,x)}$. However, as demonstrated on an example in \cite[p. 775]{fjordholm2017construction}, in contrast with entropy solutions, entropy mv solutions are not necessarily unique. 
\end{remark}
We have seen that the concept of mv solutions is weaker than the concept of weak solutions. Hence mv solutions are a relaxation of weak solutions: every weak solution is also an mv solution, but the set of mv solutions can be larger than the set of weak solutions.
However, the following result states that considering mv solutions is not a relaxation. To be more precise, when the initial measure in \eqref{eq:burgers mv: initial condition} is concentrated on (the graph of) the initial condition in \eqref{eq:burgers: initial condition}, then the entropy mv solution to \eqref{eq:burgers mv} is unique and concentrated on (the graph of) the (unique) entropy solution to \eqref{eq:burgers}.
\begin{theorem}[Concentration of the entropy mv solution]
\label{thm-concentration}
Let $C$ be the Lipschitz constant of the function $f$. Let $y$ be an entropy solution and $\mu$ be an entropy mv solution to \eqref{eq:burgers}. Then, for all $T\geq 0$ and all $r\geq 0$, it holds
\begin{equation}
\label{contraction}
\int_{|x|\leq r} \langle \mu_{(t,x)},|y-y(T,x)|\rangle dx \leq \int_{|x|\leq r+CT} \langle \measO,|y-y_0(x)|\rangle dx.
\end{equation}
In particular, if $\measO=\delta_{y_0(x)}$, then $\mu_{(t,x)}=\delta_{y(t,x)}$ for all $t \in [0,T]$ and all $x$ such that $|x|\leq r$.
\end{theorem}
\begin{remark}
\label{remark-kruzkhov}
The proof of Theorem \ref{thm-concentration} is similar to the one provided in \cite{kruvzkov1970first} and  it is postponed to Appendix \ref{proof-concentration}. It is based on the doubling variable strategy, using the following family of entropy pairs:
\begin{equation}
\label{kruzkhov-entropy}
\eta_v(y) = |y-v|,\quad q_v(y) = \sign(y-v)(f(y)-f(v))
\end{equation}
parametrized in $v \in\mathbb{R}$.
In \cite{lax1971shock}, it has been proved that linear combinations of these entropy pairs, together with the convex hull of linear functions, generate all entropy pairs. In other words, 
to prove Theorem \ref{thm-concentration} for every entropy pair it is enough to
consider the entropy pairs \eqref{kruzkhov-entropy}. 

Moreover, note that initially the doubling variable strategy has been used to prove uniqueness of the solution to scalar nonlinear conservations laws. The main drawback is that the entropy solution has to satisfy this inequality for {all} convex pairs. However
for the specific case of the Burgers equation, it is shown in \cite{de2004minimal} and \cite{panov1994uniqueness} that one may
consider only one convex pair. \end{remark} 

\subsection{An emphasis on compact sets}

In practice, one computes or approximates the solution on compact subsets, so let
\begin{equation}
\label{T-X}
\mathbf{T} : = [0,T], \quad
\mathbf{X} := [L,R]
\end{equation}
be the respective domains of time $t$ and space $x$, for fixed (but arbitrary) constants $T,L,R$.
After scaling, we assume without loss of generality that $T=R-L=1$.

Note that the entropy inequality induces a stability property:
\begin{equation}
\Vert y(t,\cdot)\Vert_{\Leb^\infty(\mathbf X)} \leq \Vert y_0\Vert_{\Leb^\infty(\mathbb{R})},\:\forall t\geq 0
\end{equation}
see e.g. \cite[Theorem 6.2.4]{dafermos2006hyperbolic}.
Since $y_0$ is bounded in $\Leb^\infty$, it follows from the maximum principle \cite[Theorem 6.3.2]{dafermos2006hyperbolic} that $y(t,.)$ is bounded in $\Leb^\infty$ for all $t\geq 0$. Hence, we can consider that $y$ takes values in the following compact set
\begin{equation}
\label{Y}
\mathbf{Y} := [\yL,\yR],
\end{equation}
where the bounds $\yL := \mathrm{ess}\:\inf_{x\in\mathbb R} y_0(x)$ and $\yR:=\mathrm{ess}\:\sup_{x\in\mathbb R} y_0(x)$ depend on the initial condition.  
On $\mathbf{T}\times\mathbf{X}$, the polynomial hyperbolic equations given in \eqref{eq:burgers} reads:
\begin{equation}
\label{eq:burgers:compact}
\left\{
\begin{split}
&\frac{\partial y}{\partial t} + \frac{\partial f(y)}{\partial x} =0, \quad (t,x)\in \mathbf{T}\times \mathbf{X},\\
&y(0,x)=y_0(x),\quad x\in\mathbf{X}.
\end{split}
\right.
\end{equation}

\begin{definition}[Entropy solution on compact sets] \label{def: emvsol on compact sets}
A weak solution $y$ is an {entropy solution} to \eqref{eq:burgers:compact} if, for all test functions $\psi_1\in \C^1(\mathbb{R}_+\times \mathbb{R})$, it satisfies
\begin{equation}
\label{1.burgers:compact}
\begin{split}
\int_{\mathbf{T}}\int_{\mathbf{X}} & \left(\frac{\partial\psi_1}{\partial t}y + \frac{\partial \psi_1}{\partial x}f(y)\right) dx dt + \int_{\mathbf{X}} \psi_1(0,x)y_0(x) dx - \int_{\mathbf{X}} \psi_1(T,x) y(T,x) dx \\
&+ \int_{\mathbf{T}} \psi_1(t,L)y(t,L) dt - \int_{\mathbf{T}} \psi_1(t,R)y(t,R) dt = 0
\end{split}
\end{equation}
and, for all convex pairs $(\eta,q)$ and all non-negative test functions $\psi_2\in \C^1(\mathbb{R}_+\times \mathbb{R})$, it satisfies
\begin{equation}
\label{2.burgers:compact}
\begin{split}
\int_{\mathbf{T}}\int_{\mathbf{X}} & \left(\frac{\partial\psi_2}{\partial t}\eta(y) + \frac{\partial \psi_2}{\partial x}q(y)\right) dx dt + \int_{\mathbf{X}} \psi_2(0,x) \eta(y_0)(x) dx- \int_{\mathbf{X}} \psi_2(T,x) \eta(y(T,x)) dx\\
& + \int_{\mathbf{T}} \psi_2(t,L)q(y(t,L)) dt - \int_{\mathbf{T}} \psi_2(t,R)q(y(t,R)) dt\geq 0.
\end{split}
\end{equation}
\end{definition}

As we work  on compact sets, the test functions do not have to vanish at infinity. However
new terms $y(t,R)$ and $y(t,L)$ now appear. Related to this notion of solutions on compact sets, we also have a similar definition for mv entropy solution.

\begin{definition}[Measure-valued entropy solution on compact sets]
A Young measure $\mu : (t,x)\in \T\times \X\mapsto \mu_{(t,x)}\in \prob(\Y)$ is an entropy measure-valued solution to \eqref{eq:burgers:compact} if, for all test functions $\psi_1\in \C^1(\T\times\X)$, it satisfies
\begin{equation}
\label{linear-burgers1}
\begin{split}
\int_{\mathbf{T}}\int_{\mathbf{X}} & \left(\frac{\partial\psi_1}{\partial t}\langle \mu_{(t,x)}, y \rangle + \frac{\partial \psi_1}{\partial x}\langle \mu_{(t,x)},f(y)\rangle \right) dx \,dt + \int_{\mathbf{X}} \psi_1(0,x)\langle \measO,y \rangle dx\\
& - \int_{\mathbf{X}} \psi_1(T,x) \langle \measT, y \rangle dx + \int_{\mathbf{T}} \psi_1(t,L)\langle \measL,f(y)\rangle dt - \int_{\mathbf{T}} \psi_1(t,R) \langle \measR, f(y) \rangle dt = 0
\end{split}
\end{equation}
and, for all convex pairs $(\eta,q)$ and all non-negative test functions $\psi_2\in \C^1(\mathbf{T}\times \mathbf{X})$, it satisfies
\begin{equation}
\label{linear-burgers2}
\begin{split}
&\int_{\mathbf{T}}\int_{\mathbf{X}}  \left(\frac{\partial\psi_2}{\partial t}\langle \mu_{(t,x)},\eta(y)\rangle + \frac{\partial \psi_2}{\partial x}\langle \mu_{(t,x)},q(y)\rangle \right) dx \,dt + \int_{\mathbf{X}} \psi_2(0,x) \langle \measO,\eta(y)\rangle dx\\
&- \int_{\mathbf{X}} \psi_2(T,x) \langle \measT,\eta(y)\rangle dx + \int_{\mathbf{T}} \psi_2(t,L)\langle \measL,q(y)\rangle dt - \int_{\mathbf{T}} \psi_2(t,R)\langle \measR, q(y)\rangle dt\geq 0,
\end{split}
\end{equation}
where $\measO$, $\measT$, $\measL$ resp. $\measR$ are Young measures supported on $\T$, $\X$, $\T$  resp. $\X$.
\end{definition}

\begin{remark}[Imposing constraints on the boundary]
\label{rem:boundary}
To ensure concentration of $\mu_{(t,x)}$ on the graph of the solution to \eqref{1.burgers:compact}-\eqref{2.burgers:compact}, in addition to the condition $\measO=\delta_{y_0(x)}$, one may impose conditions on the boundary measures $\measL$ and/or $\measR$. In practice, one knows the initial condition in an interval larger than $\X$ and
so one is able to impose $\measL$ and/or $\measR$. The width of this interval depends on the Lipschitz constant of the flux, $T$, $L$ and $R$. As an illustrative example, consider the case
where the initial condition is positive and the flux is strictly convex. By the classical method of characteristics, if the initial condition $y_0$ is positive then so is the solution $y$ for all $t\geq0$. 
In particular if $f$ is strictly convex we only need to impose knowledge at the left of the box $\X$. Therefore  $\measL$ has to be known for all $t\in \T$, and $\measR$ is unconstrained. We refer to \cite{randall1992numerical} for a more precise discussion on the choice of the boundary constraint.
\end{remark}

\section{A convex optimization approach for mv solutions on compact sets}
\label{2.convex}

In the latter section, we introduced mv solutions for scalar hyperbolic equations. Note that measures are fully characterized by their moments on compact sets, see e.g. \cite[p. 52]{lasserre2009moments}. This means in particular that moments {are} the quantities of interest. The aim of this section is to express formulations \eqref{linear-burgers1}-\eqref{linear-burgers2} as constraints on the moments, to explain how one can compute numerically these moments thanks to the moment-SOS hierarchy. We also show how one can interpret these moments in the case where the initial measure is concentrated. 

\subsection{Moment constraints for the entropy mv solution}

\label{sec_mv-gmp}

Let $\nu\in\meas(\KL)_+$, with $\KL:=\T\times\X\times \Y$. In the following, we derive moment constraints that will imply that $\nu$ can be desintegrated as follows
\begin{equation}
\label{mv-occupation-measure}
d\nu(t,x,y)=dt\,dx\,d\mu_{(t,x)}(dy)
\end{equation}
or, equivalently,
\begin{equation}
\nu =\lambda_{\T}\lambda_{\X} \mu_{(t,x)}, 
\end{equation}
where $\mu$ is an entropy mv solution satisfying \eqref{linear-burgers1} and \eqref{linear-burgers2}.
In (\ref{mv-occupation-measure}) the measure $\nu$ is called an {\it occupation measure} and the Young measure $\mu$ is its conditional measuring $y$ given $t$ and $x$.
We also need  to introduce the following time boundary measures
\begin{equation}
d\nu_0(t,x,y):=\delta_0(dt)\,dx\,\measO(dy), \quad d\nu_T(t,x,y):=\delta_T(dt)\,dx\,\measT(dy)
\end{equation}
whose supports are $\KLO :=\lbrace 0\rbrace \times\mathbf{X}\times\mathbf{Y}$ and $\KLT:=\lbrace T\rbrace\times \mathbf{X}\times\mathbf{Y}$ respectively.
Similarly, we introduce the following space boundary measures.
\begin{equation}
d\nu_{L}(t,x,y):=dt\,\delta_L(dx)\,\measL(dy), \quad d\nu_R(t,x,y):=dt\,\delta_R(dx)\,\measR(dy)\
\end{equation}
whose supports are given by $\KLL:=\T\times\lbrace L\rbrace\times \Y$ and $\KLR:=\mathbf{T}\times\lbrace R\rbrace\times\mathbf{Y}$  respectively.

First, to ensure that the marginal of $\nu$ with respect to $t$ and $x$ is the Lebesgue measure on $\T\times\X$, it suffices to impose that:
\begin{equation}
\label{moment-constraint-measm}
\int_{\KL} t^{\alpha_1}x^{\alpha_2}\,d\nu(t,x,y) = \int_{\T\times\X} t^{\alpha_1} x^{\alpha_2}\, dt\, dx,\quad \alpha \in\mathbb{N}^2.
\end{equation}
In a similar manner, we can enforce the respective marginal of the boundary measures to be products of an Dirac measure and the Lebesgue as follows
\begin{equation}
\label{moment-constraint-meas0}
\int_{\KLO} 0^{\alpha_1} x^{\alpha_2}\,d\nu_0(t,x,y) = \int_{\X} 0^{\alpha_1} x^{\alpha_2} dx,\quad \alpha\in\mathbb{N}^2,
\end{equation}
\begin{equation}
\label{moment-constraint-measT}
\int_{\KLT} T^{\alpha_1} x^{\alpha_2}\,d\nu_T(t,x,y) = \int_{\X} T^{\alpha_1} x^{\alpha_2} dx,\quad \alpha\in\mathbb{N}^2,
\end{equation}
\begin{equation}
\label{moment-constraint-measL}
\int_{\KLL} t^{\alpha_1} L^{\alpha_2}\,d\nu_L(t,x,y) = \int_{\T} t^{\alpha_1} L^{\alpha_2} dt,\quad \alpha\in\mathbb{N}^2
\end{equation}
and
\begin{equation}
\label{moment-constraint-measR}
\int_{\KLR} t^{\alpha_1} R^{\alpha_2}\,d\nu_R(t,x,y) = \int_{\T} t^{\alpha_1} R^{\alpha_2} dt,\quad \alpha\in\mathbb{N}^2.
\end{equation}

Next, we aim at proving that \eqref{linear-burgers1} and \eqref{linear-burgers2} can also be expressed by moment constraints. We split the exposition into two steps: the first one deals with \eqref{linear-burgers1}, while the second deals with \eqref{linear-burgers2}.

\subsubsection{First step: enforcing \eqref{linear-burgers1} by moment constraints}

\begin{lemma}
\label{poly-representation1}
Let $\phi^\alpha_1(t,x,y):=t^{\alpha_1}x^{\alpha_2}y$ and $\phi^\alpha_2(t,x,y):=t^{\alpha_1}x^{\alpha_2}f(y)$ for $\alpha\in\mathbb{N}^2$. Linear constraint \eqref{linear-burgers1} is equivalent to
\begin{equation}
\label{first-moment-constraint}
\int_{\KL} \left(\frac{\partial \phi^\alpha_1}{\partial t} + \frac{\partial \phi^\alpha_2}{\partial x} \right)d\nu + \int_{\KLO} \phi^\alpha_1 \, d\nu_0 - \int_{\KLT} \phi^\alpha_1 \, d\nu_T + \int_{\KLL} \phi^\alpha_2 \, d\nu_L -\int_{\KLR} \phi^\alpha_2 \, d\nu_R=0
 \end{equation}
for all $\alpha\in\mathbb{N}^2$. 
\end{lemma}

\begin{proof}\:{\bf of Lemma \ref{poly-representation1}}
Since $\T\times \X$ is a compact set, as a consequence of the Stone-Weierstrass theorem, we can restrict the test functions to
$\psi_1=t^{\alpha_1}x^{\alpha_2}$ for $\alpha \in \N^2$ to enforce \eqref{linear-burgers1}.
\end{proof}

\subsubsection{Second step: enforcing \eqref{linear-burgers2} by moment constraints}

As noticed in Remark \ref{remark-kruzkhov}, the entropy inequality is satisfied for all convex pairs $(\eta,q)$ if and only it is satisfied for all Kruzkhov entropies given in \eqref{kruzkhov-entropy}. To express \eqref{linear-burgers2} as moment constraints, we are faced with two issues: first, taking into account an uncountable family of functions parametrized by $v\in\Y$ and, second, the absolute value function $v\mapsto \vert v\vert$ is not a polynomial. To deal with the uncountable family of functions, we introduce $v$ as a new variable. To treat the absolute value, we double the number of measures. 

More precisely, we define the Borel measures $\vartheta^+,\vartheta^-$ whose supports are defined as follows
\begin{equation*}
\ELp:=\spt{\vartheta^+} = \lbrace (t,x,y,v)\in \KL \times\mathbf{Y}: y\geq v\rbrace, 
\end{equation*} 
\begin{equation*}
\ELm:=\spt{\vartheta^-} = \lbrace (t,x,y,v)\in \KL \times\mathbf{Y}: y\leq v\rbrace.
\end{equation*}

Similarly, we define the time boundary measures $\vartheta_0^+$, $\vartheta_0^-$, $\vartheta_T^+$ and $\vartheta_T^-$ with the following supports
\begin{equation}
\begin{split}
&\ELOp:=\spt{\vartheta^+_0}=\lbrace (t,x,y,v)\in \KLO\times \mathbf{Y}: y\geq v\rbrace,\\
&\ELOm:=\spt{\vartheta^-_0}=\lbrace (t,x,y,v)\in \KLO\times \mathbf{Y}: y\leq v\rbrace,
\end{split}
\end{equation}
and
\begin{equation}
\begin{split}
&\ELTp:=\spt{\vartheta^+_T}=\lbrace (t,x,y,v)\in\KLT\times \mathbf{Y}: y\geq v\rbrace,\\
&\ELTm:=\spt{\vartheta^-_T}=\lbrace (t,x,y,v)\in \KLT\times\mathbf{Y}: y \leq v\rbrace.
\end{split}
\end{equation}

Finally, let us define the space boundary measures $\vartheta^+_L$, $\vartheta_L^-$, $\vartheta_R^+$ and $\vartheta_R^-$ with the following supports
\begin{equation}
\begin{split}
&\ELLp:=\spt{\vartheta^+_L}=\lbrace (t,x,y,v)\in\KLL\times \mathbf{Y}: y\geq v\rbrace\\
&\ELLm:=\spt{\vartheta^-_L}=\lbrace (t,x,y,v)\in\KLL\times \mathbf{Y}: y\leq v\rbrace
\end{split}
\end{equation}
and
\begin{equation}
\begin{split}
&\ELRp:=\spt{\vartheta^+_R}=\lbrace (t,x,y,v)\in\KLR\times \mathbf{Y}: y\geq v\rbrace,\\
&\ELRm:=\spt{\vartheta^-_R}=\lbrace (t,x,y,v)\in\KLR\times \mathbf{Y}: y\leq v\rbrace.
\end{split}
\end{equation}

We are now in position to state the following lemma.
\begin{lemma}[Recovering all Kruzkhov entropies]
\label{lemma-recover}
Assume that
\begin{equation}
\label{constraint-k1}
\vartheta^+ + \vartheta^- = \nu \otimes \lambda_{\mathbf{Y}}, 
\end{equation}
\begin{equation}
\label{constraint-k2}
\vartheta^+_0 + \vartheta^-_T = \nu_0\otimes \lambda_{\mathbf{Y}}, \quad \vartheta^+_T  + \vartheta^-_T = \nu_T\otimes \lambda_{\mathbf{Y}},
\end{equation}
\begin{equation}
\label{constraint-k3}
\vartheta^+_L + \vartheta^-_L = \nu_L\otimes \lambda_{\mathbf{Y}}, \quad \vartheta^+_R + \vartheta^-_R = \nu_R\otimes \lambda_{\mathbf{Y}}.
\end{equation}
Then, \eqref{linear-burgers2} is equivalent to
\begin{equation}
\label{2.burgers:compact:poly}
\begin{split}
&\int_{\ELp} \theta(v)\left(\frac{\partial \psi_2}{\partial t}(y-v) + \frac{\partial \psi_2}{\partial x} (f(y)-f(v))\right) d\vartheta^+ \\
&+\int_{\ELm} \theta(v) \left(\frac{\partial \psi_2}{\partial t}(v-y) + \frac{\partial \psi_2}{\partial x} (f(v)-f(y))\right) d\vartheta^-\\
&+ \int_{\ELOp} \theta(v)\psi_2(0,x)(y-v) d\vartheta_0^+ +\int_{\ELOm}\theta(v)\psi_2(0,x)(v-y) d\vartheta_0^- \\
&- \int_{\ELTp}\theta(v)\psi_2(T,x)(y-v) d\vartheta_T^+ - \int_{\ELTm}\theta(v)\psi_2(T,x)(v-y) d\vartheta_T^-\\
&+\int_{\ELLp} \theta(v)\psi_2(t,L)(f(y)-f(v)) d\vartheta_L^+ + \int_{\ELLm} \theta(v)\psi_2(t,L)(f(v)-f(y)) d\vartheta_L^- \\
&- \int_{\ELRp} \theta(v) \psi_2(t,R)(f(y)-f(v)) d\vartheta_R^+ - \int_{\ELRm} \theta(v) \psi_2(t,R)(f(v)-f(y) ) d\vartheta_R^-\geq 0,
\end{split}
\end{equation}
for all nonnegative functions $\theta\in \C(\Y)$. 
\end{lemma}

Note that from the Stone-Weierstrass Theorem, the constraints \eqref{constraint-k1}, \eqref{constraint-k2} and \eqref{constraint-k3} can be expressed as moment constraints: \eqref{constraint-k1} holds if and only if, for all $ \alpha\in\N^4$,
\begin{equation}
\label{moment-constraint-meas}
\int_{\ELp\cup \ELm} t^{\alpha_1}x^{\alpha_2}y^{\alpha_3} v^{\alpha_4} d(\vartheta^+ + \vartheta^-)(t,x,y,v) =  \int_{\KL} t^{\alpha_1}x^{\alpha_2}y^{\alpha_3} d\nu(t,x,y) \int_{\Y} v^{\alpha_4}  dv
\end{equation}
and similarly for \eqref{constraint-k2} and \eqref{constraint-k3}.

\begin{proof}\textbf{ of Lemma \ref{lemma-recover}:}
For conciseness, we focus only on the first two term in \eqref{2.burgers:compact:poly}. The terms considering the boundary measures can be treated similarly. Let us prove that if for all nonnegative functions $\theta\in \C^1(\mathbf{Y})$ and all nonnegative functions $\psi_2\in \C^1(\mathbf{T}\times\mathbf{X})$,
\begin{equation}
\label{proof-kruzkhov}
\begin{split}
&\int_{\ELp} \theta(v)\left(\frac{\partial \psi_2}{\partial t}(y-v) + \frac{\partial \psi_2}{\partial x} (f(y)-f(v))\right) d\vartheta^+ \\
&+\int_{\ELm} \theta(v) \left(\frac{\partial \psi_2}{\partial t}(v-y) + \frac{\partial \psi_2}{\partial x} (f(v)-f(y))\right) d\vartheta^-=\\
&\int_{\ELp\cup\ELm} \theta(v)\left(\frac{\partial \psi_2}{\partial t}|y-v| + \frac{\partial \psi_2}{\partial x}\sign(y-v)(f(y) -f(v))\right)d(\vartheta^+ +\vartheta^-)\geq 0
\end{split}
\end{equation}
then the following inequality holds, for all test functions $\psi_2\in \C^1(\mathbf{T}\times\mathbf{X})$ and all $v\in\mathbf{Y}$:
\begin{equation}
\int_{\mathbf{T}\times\mathbf{X}\times\mathbf{Y}} \frac{\partial \psi_2}{\partial t}|y-v| + \frac{\partial\psi_2 }{\partial x}\sign(y-v)(f(y)-f(v))d\nu \geq 0.
\end{equation}
First, observe that \eqref{constraint-k1} implies that
\begin{equation}
\begin{split}
&\int_{\ELp \cup \ELm} \theta(v)\left(\frac{\partial \psi_2}{\partial t}|y-v| + \frac{\partial \psi_2}{\partial x}\sign(y-v)(f(y)-f(v))\right)d(\vartheta^+ +\vartheta^-)=\\
&\int_{\mathbf{Y}} \theta(v)\left( \int_{\mathbf{T}\times\mathbf{X}\times\mathbf{Y}}\left(\frac{\partial \psi_2}{\partial t}|y-v| + \frac{\partial \psi_2}{\partial x}\sign(y-v)(f(y)-f(v))\right) d\nu \right) dv.
\end{split}
\end{equation}
Then, since \eqref{proof-kruzkhov} holds for any nonnegative functions $\theta$, 
and $y-v=\vert y-v\vert$ on $\spt{\vartheta^+}$ 
(resp. $v-y=\vert y-v\vert$ on $\spt{\vartheta^-}$),
\begin{equation}
 \int_{\mathbf{T}\times\mathbf{X}\times\mathbf{Y}}\left(\frac{\partial \psi_2}{\partial t}|y-v| + \frac{\partial \psi_2}{\partial x}\sign(y-v)(f(y)-f(v))\right) d\nu \geq 0.
\end{equation}
\end{proof}

In order to express \eqref{linear-burgers2} as moment constraints, it remains to prove that the functions $\psi_2$ and $\theta$ can be replaced by suitable polynomials. Here, in contrast with the first step, where the functions $\psi_1$ were unconstrained, the functions $\psi_2$ and $\theta$ have to be nonnegative. To address this issue, we again use positivity certificates from real algebraic geometry.

\begin{lemma}
\label{lemma-poly-representation-2}
Let \begin{equation}
\begin{split}
&\phi^{+,\alpha}_1(t,x,y,v):=t^{\alpha_1}(T-t)^{\alpha_2}(x-L)^{\alpha_3}(R-x)^{\alpha_4}(v-\yL)^{\alpha_5}(\yR-v)^{\alpha_6}(y-v),\\
&\phi^{-,\alpha}_1(t,x,y,v):=t^{\alpha_1}(T-t)^{\alpha_2}(x-L)^{\alpha_3}(R-x)^{\alpha_4}(v-\yL)^{\alpha_5}(\yR-v)^{\alpha_6}(v-y),\\
&\phi^{+,\alpha}_2(t,x,y,v)::=t^{\alpha_1}(T-t)^{\alpha_2}(x-L)^{\alpha_3}(R-x)^{\alpha_4}(v-\yL)^{\alpha_5}(\yR-v)^{\alpha_6}(f(y)-f(v)),\\
&\phi^{-,\alpha}_2(t,x,y,v):=t^{\alpha_1}(T-t)^{\alpha_2}(x-L)^{\alpha_3}(R-x)^{\alpha_4}(v-\yL)^{\alpha_5}(\yR-v)^{\alpha_6}(f(v)-f(y))
\end{split}
\end{equation}
for $\alpha\in\N^6$.
Then, \eqref{linear-burgers2} is equivalent to
\begin{equation}
\label{second-moment-constraint}
\begin{split}
&\int_{\ELp} \left(\frac{\partial \phi^{+,\alpha}_1}{\partial t} + \frac{\partial \phi^{+,\alpha}_2}{\partial x}\right) d\vartheta^+ + \int_{\ELm} \left(\frac{\partial \phi^{-,\alpha}_1}{\partial t} + \frac{\partial \phi^{-,\alpha}_2}{\partial x}\right) d\vartheta^{-}\\
& + \int_{\ELOp} \phi^{+,\alpha}_1 d\vartheta_0^+ + \int_{\ELOm} \phi^{-,\alpha}_1 \d\vartheta_0^- - \int_{\ELTp} \phi^{+,\alpha}_1 d\vartheta_T^+ - \int_{\ELTm} \phi^{-,\alpha}_1 \d\vartheta_T^-\\
& + \int_{\ELLp} \phi^{+,\alpha}_2 d\vartheta_L^+ + \int_{\ELLm} \phi^{-,\alpha}_2 d\vartheta_L^- - \int_{\ELRp} \phi^{+,\alpha}_2 d\vartheta_R^+ - \int_{\ELRm} \phi^{-,\alpha}_2 d\vartheta_R^-\geq 0
\end{split}
\end{equation}
for all $\alpha\in\N^6$.
\end{lemma}

\begin{proof}\textbf{ of Lemma \ref{lemma-poly-representation-2}:}
The proof relies on a result of real algebraic geometry. 
Again, invoking the Stone-Weierstrass Theorem, in \eqref{linear-burgers2} we can restrict the test functions $\psi_2$ and $\theta$ to be polynomials. To enforce their positivity, we use Handelman's Positivstellensatz \cite{handelman1988} that implies that
\begin{equation}
\label{Handelman}
\begin{split}
&\psi_2(t,x) = \sum_{\alpha_1,\alpha_2,\alpha_3,\alpha_4\in\mathbb{N}^4} c^{\psi_2}_{\alpha} t^{\alpha_1}(T-t)^{\alpha_2}(x-L)^{\alpha_3}(R-x)^{\alpha_4},\\
&\theta(v) = \sum_{\alpha_5,\alpha_5\in\mathbb{N}} c^{\theta}_{\alpha} (v-\yL)^{\alpha_5}(\yR-v)^{\alpha_6}
\end{split}
\end{equation}
with finitely many positive real coefficients $c^{\psi_2}_{\alpha}, c^{\theta}_{\alpha}$. Now, \eqref{second-moment-constraint} implies that
\begin{equation}
\label{second-moment-constraint-non}
\begin{split}
&\sum_{\alpha\in\mathbb{N}^{6}}c^{\psi_2}_{\alpha} c^{\theta}_{\alpha} \left\{\int_{\ELp} \left(\frac{\partial \phi^{+,\alpha}_1}{\partial t} + \frac{\partial \phi^{+,\alpha}_2}{\partial x}\right) d\vartheta^+ + \int_{\ELm} \left(\frac{\partial \phi^{-,\alpha}_1}{\partial t} + \frac{\partial \phi^{-,\alpha}_2}{\partial x}\right) d\vartheta^{-} \right.\\
& + \int_{\ELOp} \phi^{+,\alpha}_1 d\vartheta_0^+ + \int_{\ELOm} \phi^{-,\alpha}_1 d\vartheta_0^- - \int_{\ELTp} \phi^{+,\alpha}_1 d\vartheta_T^+ - \int_{\ELTm} \phi^{-,\alpha}_1 d\vartheta_T^- \\
& + \int_{\ELLp} \phi^{+,\alpha}_2 d\vartheta_L^+ + \int_{\ELLm} \phi^{-,\alpha}_2 d\vartheta_L^- \left.- \int_{\ELRp} \phi^{+,\alpha}_2 d\vartheta_R^+ - \int_{\ELRm} \phi^{-,\alpha}_2 d\vartheta_R^-\right\}\geq 0,
\end{split}
\end{equation}
which, by linearity of the integrals and the derivatives, recovers  \eqref{2.burgers:compact:poly} for $\psi_2$ and $\theta$ given in \eqref{Handelman}. Consequently, by Lemma \ref{lemma-recover}, \eqref{2.burgers:compact:poly} implies that \eqref{linear-burgers2} holds.\end{proof}

\begin{remark}
Note that the measure defined in \eqref{mv-occupation-measure} is similar to the occupation measure introduced in \cite{lasserre2008nonlinear}, which deals with optimal control of nonlinear finite-dimensional systems.  This notion has been further used in many other contexts, as for instance the computation of region of attraction \cite{henrion2014convex}. Therefore, the formulation given in \eqref{linear-burgers1}-\eqref{linear-burgers2} might be instrumental to solve other problems than computing numerically the solution of scalar hyperbolic PDE. 
\end{remark}

The next section aims at showing that a moment formulation can be numerically solved thanks to the moment-SOS hierarchy and SDP.

\subsection{The Generalized Moment Problem and its relaxations}

Roughly speaking, the {\em Generalized Moment Problem} (GMP) is an infinite-dimensional linear optimization problem on finitely many Borel measures $\optmeas_i\in\meas(\K_i)_+$ whose supports are contained in given sets $\K_i\subseteq\R^{n_i}$, with $i=1,\ldots,k$ and $n_i\in\N$. That is, one is interested in finding measures whose moments satisfy  (possibly countably many) linear constraints and which minimize a linear criterion. In full generality, the GMP is intractable, but if all
$\K_i$ are basic semi-algebraic sets and the integrands are polynomials (semi-algebraic functions are also allowed\footnote{A semi-algebraic function is a function whose graph is a semi-algebraic set, i.e. it is described by finitely many polynomial inequalities and equations.}), then one may provide an efficient numerical scheme
to approximate as closely as desired any finite number of moments 
of optimal solutions of the GMP. It consists of solving a hierarchy of 
semidefinite programs\footnote{A semidefinite program is a particular class of a convex conic optimization problem that can be solved numerically efficiently.} of increasing size. Convergence of
this numerical scheme is guaranteed by invoking
powerful results from Real Algebraic Geometry (essentially positivity certificates). 

Let $h_i\in \mathbb{R}[\genvar^i]$ and $h_{i,k} \in\mathbb{R}[\genvar^i]$ be polynomials in the vector of indeterminates $\genvar^i \in \R^{n_i}$ and let $b_k$ be real numbers, for finitely many $i=1,2,\ldots,N$ and countably many $k=1,2,\ldots$. The GMP is the problem
\begin{equation}\label{gmp}
\begin{array}{lll}
\rho^\star:= & \inf_{\optmeas} & \sum_{i=1}^N\int_{\K_i} h_i d\nu_i \\
& \mathrm{s.t.} & \sum_{i=1}^{N}\int_{\K_i} h_{i,k} d\optmeas_i \leqq b_k,\quad k=1,2,\ldots\\
& & \optmeas_i\in\meas(\K_i)_+,\quad i=1,\ldots,N.
\end{array}
\end{equation}

\textbf{Entropy mv solution as a GMP} In the scalar hyperbolic case, the measures $\nu_i$ under consideration are $\nu, \nu_T, \nu_0, \nu_R, \nu_L$ and all the measures we have introduced when transforming the Kruzkhov inequality into moment constraints. The sets $\K_i$  correspond to $\T$, $\X$ and $\Y$. Finally, the polynomials $h_{i,j}$ are given in \eqref{first-moment-constraint} (conservation law) \eqref{second-moment-constraint} (entropy inequality), \eqref{moment-constraint-meas} (Kruzhkov entropies), and (\ref{moment-constraint-measm}-\ref{moment-constraint-measR}) (boundary measures).

We may also define an objective functional
\begin{equation}
\label{cost-function}
\int_{\KL} h d\nu + \int_{\KLO} h_0 d\nu_0 + \int_{\KLT} h_Td\nu_T + \int_{\KLL} h_L d\nu_L + \int_{\KLR} h_Rd\nu_R,
\end{equation}
with $h,h_0,h_T,h_L,h_R\in\mathbb{R}[t,x,y]$.

If $\sigma_0 = \delta_{y_0(x)}$ with $y_0$ an initial condition in \eqref{1.burgers:compact}-\eqref{2.burgers:compact} and, in addition, if one imposes suitable boundary measures as exposed in Remark \ref{rem:boundary}, then this objective functional is not especially useful to recover the entropy mv solution of scalar hyperbolic PDE, since the corresponding Young measure is concentrated as a consequence of Theorem \ref{thm-concentration}: there is nothing to be optimized. However, with such an objective functional, one can compute quantities of interest such as the energy of the solution. Moreover, our aim is to \textit{relax} the GMP in order to solve it numerically and, then, this objective functional might be helpful to accelerate the convergence of the corresponding relaxations. We refer to Section \ref{3.riemann} for more discussions about the choice of objective functionals for the Riemann problem of the Burgers equation.

Finally, one is able to define a GMP describing entropy mv solution:
\begin{equation}
\label{gmp-pde}
\begin{array}{ll}
\inf_{\nu,\: \nu_T,\: \nu_L\text{ and/or }\nu_R} & \eqref{cost-function} \text{ (objective functional)}\\
\text{s.t.} & \eqref{first-moment-constraint} \text{ (conservation law)},\\
& \eqref{second-moment-constraint} \text{ (entropy inequality)},\\
& \eqref{moment-constraint-meas} \text{ (Kruzhkov entropies)},\\ & (\ref{moment-constraint-measm}-\ref{moment-constraint-measR}) \text{ (boundary measures)},\\
& \nu\in\meas(\KL)_+ \text{ (occupation measure)},\\
& \nu_T\in\meas(\KLT)_+ \text{ (time boundary measure)},\\
& \nu_L\in\meas(\KLL)_+ \text{ (space boundary measure)},\\
& \text{and/or } \nu_R\in\meas(\KLR)_+ \text{ (space boundary measures)}
\end{array}
\end{equation}
where the measures defined in \eqref{constraint-k1}-\eqref{constraint-k3} and related to the Kruzkhov entropies are considered as implicit variables.

\textbf{From measures to moments} Instead of optimizing over the measures in problem (\ref{gmp-pde}), we optimize over their moments. For simplicity and clarity of exposition, we describe the approach 
in the case of a single unknown measure $\nu$, but it easily extends to the case of several measures. So consider the simplified GMP:
\begin{equation}\label{gmp-simple}
\begin{array}{lll}
\rho^\star = & \inf_{\optmeas} & \int_{\K} h d\nu \\
& \text{s.t.} & \int_{\K} h_k d\optmeas \leqq b_k,\quad k=1,2,\ldots\\
& & \optmeas\in\meas(\K)_+.
\end{array}
\end{equation}
The \emph{moment sequence} $\mom = (\mom_\alpha)_{\alpha\in\N^n}$ of a given measure $\optmeas\in\meas(\K)_+$ is defined by 
\begin{equation}\label{moments}
\mom_\alpha=\int_K \genvar^\alpha \d\optmeas, \quad \alpha\in\N^n
\end{equation}
where $\genvar^\alpha =w_1^{\alpha_1}\cdot\ldots\cdot w_n^{\alpha_n}$. Conversely, given a sequence $(\mom_\alpha)_{\alpha\in\N^n}$, if (\ref{moments}) holds for some $\optmeas\in\meas(\K)_+$ we say that the sequence has the representing measure $\optmeas$. Recall that measures on compact sets are uniquely characterized by their moments; see e.g.
\cite[p. 52]{lasserre2009moments}. 

Let $\N^n_d:=\{\alpha\in\N^n : |\alpha|\leq d\}$, where $|\alpha|:=\sum_{i=1}^n \alpha_i$,
and $s(d):=\binom{n+d}{d}$.
A vector $\p:=(\p_\alpha)_{\alpha\in\mathbb{N}^n_d}\in\R^{s(d)}$is the {coefficient vector} (in the monomial basis) of a polynomial $p\in\rx$  with $d=\deg(p)$  expressed as
$p = \sum_{\alpha\in\N^n_{\deg(p)}} \p_\alpha \genvar^\alpha$. Next, integration of 
$p$ with respect to a measure $\optmeas$ involves only finitely many moments:
\[
\int_\K p \d\optmeas = \int_\K \sum_{\alpha\in\N^n_d} \p_\alpha \genvar^\alpha\d\optmeas =  \sum_{\alpha\in\N^n_d} \p_\alpha \int_\K \genvar^\alpha\d\optmeas =  \sum_{\alpha\in\N^n_d} \p_\alpha\, \mom_\alpha.
\]
Next, define a {pseudo-integration} with respect to an arbitrary sequence $\mom\in\R^{\N^n}$ by:
\begin{equation}\label{defeq:Riesz functional}
\rf_\mom(p):= \sum_{\alpha\in\N^n} \p_\alpha z_\alpha
\end{equation}
and $\rf_\mom$ is called the \emph{Riesz functional}. Moment sequences
can be characterized via the Riesz functional:
\begin{theorem}[Riesz-Haviland {\cite[Theorem 3.1]{lasserre2009moments}}]
\label{thm:riesz}
Let $\K\subseteq\R^n$ be closed. A real sequence $\mom$ is the sequence of some measure $\optmeas$ supported on $\K$ if and only if $\rf_\mom(p)\geq 0$ for all $p\in\rx$ nonnegative on $\K$. 
\end{theorem}
Assuming that $\K$ is closed, we can reformulate the GMP \eqref{gmp-simple} as a linear problem on moment sequences. Consider the optimization problem:
\begin{equation}\label{defeq:GMP_MOM}
\begin{array}{lll}
\rho^\ast= & \inf_{\mom} & \rf_\mom(h) \\
& \text{s.t.} & \rf_\mom(h_k)\leqq b_k,\:k=1,2,\ldots \\
&& \rf_\mom(p)\geq 0,\: \text{for all}\:p\in\rx\:\text{nonnegative on}\:\K.
\end{array}
\end{equation}
By Theorem \ref{thm:riesz}, the two formulations \eqref{defeq:GMP_MOM}
and \eqref{gmp-simple} are equivalent. Of course problem \eqref{defeq:GMP_MOM}
is still numerically intractable. 

The second and last step to approximate GMPs numerically consists of replacing the cone of polynomials nonnegative on $\K$ by a more tractable cone. This is where one exploits the fact that $\K$ is basic semi-algebraic set.

\textbf{From nonnegative polynomials to sums of squares} Characterizing nonnegativity of polynomials is an important issue in real algebraic geometry. 
Let $\K$ be a {\em basic semi-algebraic set}, i.e.:
\begin{equation}
\label{semi-algebraic-set}
\K=\{\genvar\in\R^n : g_1(\genvar)\geq0,\ldots,g_m(\genvar)\geq0\}.
\end{equation}
for some polynomials $g_1,\ldots,g_m\in\rx$, and assume that $\K$ is compact. In addition
assume  that one of the polynomials, i.e. the first one, is $g_1(\genvar):=N-\sum_{i=1}^n\genvar_i^2$ for some $N\in\N$ sufficiently large\footnote{This condition is slightly stronger than asking $\K$ to be basic semi-algebraic compact. However, the inequality $N-\sum_{i=1}^n\genvar_i^2\geq 0$ can always be added as a redundant constraint to the description of a basic semi-algebraic compact set.}. For notational convenience we let $g_0(\genvar) := 1$.

\begin{remark}

Note that the compact sets $\T,\X$ and $\Y$ which are defined in the latter section can be expressed as basic semi-algebraic compact sets. Indeed, one has
\begin{equation}
\T=\lbrace t\in\R: t(T-t)\geq 0\rbrace,\quad \X=\lbrace x\in\R: (x-L)(R-x)\geq 0\rbrace,\quad \Y=\lbrace y\in\R: (y-\yL)(\yR-y)\geq 0\rbrace.
\end{equation}
\end{remark}
Recall that a polynomial $s\in\rx$ is a \emph{sum of squares} (SOS) if 
there are finitely many polynomials $q_1,\ldots,q_r$ such that
$s(\genvar)=\sum_{j=1}^{r} q_j(\genvar)^2$ for all $\genvar$.

The following result due to Putinar \cite{putinar1993positive} is crucial
to approximate \eqref{defeq:GMP_MOM} numerically.
\begin{theorem}[Putinar's Positivstellensatz]\label{thm:putinar}
If $p>0$ on $\K$ then
$p = \sum_{j=0}^m s_jg_j$
for some SOS polynomials $s_j\in\rx$, $j=0,1,\ldots,m$.
\end{theorem}
By a density argument, checking nonnegativity of $\rf_\mom$ on polynomials nonnegative on $\K$ can be replaced by checking nonnegativity only on polynomials that are strictly positive on $\K$ and hence on those that have an SOS representation as in Theorem \ref{thm:putinar}. 

Next, for a given integer $d$, 
denote by $\Sigma[\genvar]_d\subset\R[\genvar]$ the set of SOS polynomials 
of degree at most $2d$, and define the cone $Q(g) \subset \R[\genvar]$ by:
\begin{equation}\label{quad-module}
Q_d(g)\,:=\,\left\{\sum_{j=0}^m\sigma_j\,g_j: 
\:{\rm deg}(\sigma_j\,g_j)\leq 2d,\:\sigma_j\in\Sigma[\genvar],\: j=0,1,\ldots,m\right\}
\end{equation}
and observe that $Q_d(g)\subset Q_{d+1}(g)$ consist of polynomials positive on $\K$ for all $d$.

Let $\mmon_d :=(\genvar^\alpha)_{|\alpha|\leq d} \in \R[\genvar]^{s(d)}$ be the vector of monomials of degree at most $d$. For instance, for $n=2$ and $d=3$, $\mmon_d = (1 \: w_1\: w_2 \: w_1^2 \; w_1w_2 \; w_2^2 \; w_1^3 \; w_1^2w_2 \; w_1w_2^2 \; w_2^3)$.
For $j=0,1,\ldots,m$, let $d_j$ denote the smallest integer larger than or equal to $\deg(g_j)/2$, let
$M_{d-d_j}(g_j\:\mom)$ denote the real symmetric matrix linear in $\mom$ corresponding to the entrywise application of $\rf_\mom$ to the matrix polynomial $g_j\,\mmon_{d-d_j}\mmon_{d-d_j}^T$. For $j=0$, i.e., $g_0 = 1$, this matrix is called \emph{moment matrix}.
It turns out that $\rf_\mom(g_j\,q^2)\geq0$ for all $q\in\R[\genvar]_d$ if and only if $M_{d-d_j}(g_j\:\mom)\succeq 0$ where the inequality means positive semidefinite. Therefore checking whether $\rf_\mom$ is nonnegative on 
$Q_d(g)$ reduces to checking whether $M_{d-d_j}(g_j\:\mom)\succeq0$ for $j=0,1,\ldots,m$, which are convex linear matrix inequalities in $\mom$.

\textbf{Moment-SOS hierarchy}  The following finite-dimensional semidefinite programming (SDP) problems are relaxations of the moment problem \eqref{defeq:GMP_MOM}:
\begin{equation}\label{defeq:SDPrelaxation}
\begin{array}{lll}
\rho_d^\ast= & \inf_{\mom} & \rf_\mom(h) \\
& \text{s.t.} & \rf_\mom(h_k)\leqq b_k, \;\text{deg}(h_k) \leq 2d,\:k=1,2,\ldots\\
&& M_{d-d_j}(g_j\:\mom)\succeq 0, \: j=0,1,\ldots,m
\end{array}
\end{equation}
and they are parametrized by the relaxation order $d$.

\begin{theorem}[Convergence of the moment-SOS hierarchy \cite{lasserre2009moments}]
\label{thm:lasserre-hierarchy}
 Assume there is some $M>0$ and $k\in\N$ such that $\rf_\mom(h_k)\leq b_k$ implies $\rf_\mom(1)<M$.
 Then:
 \begin{itemize}
\item[(i)] The semidefinite relaxation \eqref{defeq:SDPrelaxation} has an optimal solution
$\mom^d=(\mom^d_{\alpha})$, $\rho_d\leq \rho_{d+1}$ and $\lim_{d\to\infty}\rho_d = \rho^\ast$;
\item[(ii)] If \eqref{defeq:GMP_MOM} has a unique minimizer $\mom^\ast$, then
\begin{equation}
\label{conv}
\lim_{d\to\infty}\mom^d_\alpha\,=\,\mom^\ast_\alpha,\quad\forall\,\alpha\in\N^n.
\end{equation}
\end{itemize}
\end{theorem}

A proof is provided in \cite{lasserre2009moments}, but, for clarity, we recall the steps and the arguments used to obtain the result.

\begin{proof}\textbf{ of Theorem \ref{thm:lasserre-hierarchy}:}
Let $\mom:=(\mom_\alpha)_{\alpha\in\N^n_{2d}}$ be a feasible solution
of \eqref{defeq:SDPrelaxation}. From $M_{d-d_1}(g_1\,\mom)\succeq0$
we obtain $\sum_{i^=1}^n \rf_\mom(\genvar_i^2)\leq N\mom_0$, and in particular
$\rf_\mom(\genvar_i^2)\leq N\mom_0$, $i=1,\ldots,n$. 
By iterating one also obtains $\rf(\genvar_i^{2d})\leq N^d\mom_0$, 
$i=1,\ldots,n$.  
Moreover, combining with $M_d(\mom)\succeq0$, from \cite{lasserre2009moments}:
\begin{equation*}
|\mom_\alpha |\leq \max\left[\mom_0, \max_{i=1,\ldots,n} \rf_{\mom}(\genvar_i^{2d})\right]\,=:\,\tau_d,\quad \forall \alpha\in\mathbb{N}^{n}_{2d}.
\end{equation*} 
This inequality together with the fact that $\rf_\mom(1)=\mom_0$  is bounded implies that the moment sequence $\mom$ is uniformly bounded. Then, the feasible set of \eqref{defeq:SDPrelaxation} is closed, bounded, and hence compact. Hence \eqref{defeq:SDPrelaxation} has an optimal solution $\mom^d$.

Next, for a given $d$, let $\mom^d$ be an optimal solution to \eqref{defeq:SDPrelaxation} and 
complete $\mom^d$ with zeros to make it an infinite sequence indexed by $\alpha\in\N^n$.
Then define:
\begin{equation}
\label{scaling}
\hat{\mom}^d_\alpha:=\frac{\mom^d_\alpha}{\tau_k},\quad \forall\alpha\in\mathbb{N}^n;\quad
2k-1\leq\vert\alpha\vert\leq 2k;\:k=1,\ldots,d.
\end{equation} 
By construction, $|\hat{\mom}^d_{\alpha}|\leq 1$, for all $\alpha\in\N^n$ and therefore
$\hat{\mom}^d$ becomes an element of the unit ball $\mathbf{B}_1$ of the Banach space $\ell_\infty$ of bounded sequences, equipped with the sup-norm. Since $\ell_\infty$ is the topological dual of $\ell_1$, by the Banach-Alaoglu theorem \cite[Theorem 3.16]{brezis2010functional}, $\mathbf{B}_1$ is weak star (sequentially) compact. Hence there exists $\hat{\mom}^\star\in\mathbf{B}_1$ and a subsequence $\lbrace d_k\rbrace\subset \N$ such that $\hat{\mom}^{d_k}\rightarrow \hat{\mom}^\star$ for the weak star topology $\sigma(\ell_\infty,\ell_1)$. In particular, for every $\alpha\in\mathbb{N}^n$, $\lim_{k\rightarrow \infty} \hat{\mom}^{d_k}_\alpha=\hat{\mom}^\star_\alpha$.  Since $\tau_k$ is bounded for all $k=1,\ldots,d$, using \eqref{scaling} 
in the other direction, there exists $\mom^\ast=(\mom^\ast_\alpha)_{\alpha\in\N^n}$ such that 
\begin{equation}
\label{pointwise-convergence}
\lim_{k\rightarrow \infty} \mom_\alpha^{d_k}=\mom^\star_\alpha,\quad \forall \alpha\in\mathbb{N}^n.
\end{equation}
The pointwise convergence \eqref{pointwise-convergence} implies $M_d(g_j \:\mom^\star)\succeq 0$ for $j=0,1,\ldots,m$. Hence by Theorem \ref{thm:putinar}, $\mom^\ast$ has a representing measure $\nu$ supported on $\K$. In particular, from \eqref{pointwise-convergence},
$\rf_\mom(h_k)\leqq b_k$ for all $k=1,2,\ldots$ which proves that $\nu$ is a feasible solution of \eqref{gmp-simple}. In addition, 
\begin{equation}
\rho^\star \geq \lim_{k\rightarrow +\infty}\rho_{d_k}^\star=\int_{\K} h d\nu\geq \rho^\star,
\end{equation}
which proves that $\nu$ is an optimal solution of \eqref{gmp-simple}. Finally, if
\eqref{gmp-simple} has a unique minimizer $\nu^\star$ then 
$\nu=\nu^\star$ and the convergence \eqref{pointwise-convergence} holds for the whole sequence,
which yields \eqref{conv}. This concludes the proof. \end{proof}

\begin{remark}[Extension to several measures]
\label{extension-several}
Theorem \ref{thm:lasserre-hierarchy} extends naturally to the GMP 
\eqref{gmp} with finitely many measures. A more detailed discussion is provided in \cite[Section 4.5.2 p. 88]{lasserre2009moments}.
\end{remark}

\textbf{Convergence of the relaxations of \eqref{gmp-pde}} Problem \eqref{gmp-pde} can be approximated by a hierarchy of semidefinite relaxations as mentioned in Remark \ref{extension-several}. Moreover observe that the mass of all measures appearing in \eqref{gmp-pde} is bounded, because their marginals with respect to time and/or space are Lebesgue. Indeed, for instance, $\int_{\T\times\X\times\Y} d\nu=\int_{\T\times \X \times \Y} dt\,dx\,\mu_{(t,x)}(dy)=1$, where we have used the fact that $\mu$ is a Young measure. 

Then, according to Remark \ref{extension-several} and Theorem \ref{thm:lasserre-hierarchy}, optimal solutions of the moment-SOS hierarchy \eqref{defeq:SDPrelaxation} (adapted to the present context) converge to optimal solutions of \eqref{gmp-pde} as $d$ goes to infinity, in the sense of \eqref{conv}. In particular, one may extract the mv solution of \eqref{linear-burgers1} and even obtain the entropy solution of \eqref{1.burgers:compact}-\eqref{2.burgers:compact}, provided that $\sigma_L$ and/or $\sigma_R$ and $\sigma_0$ are concentrated.

\subsection{Interpretation of the moment solutions}
\label{sec_extract}
An optimal solution $\mom^d$ at step $d$ of the moment-SOS hierarchy of relaxations \eqref{defeq:SDPrelaxation} adapted to
the GMP \eqref{gmp-pde}, consists of  finite sequences of approximate moments, one for each unknown measure 
of \eqref{gmp-pde}. If one is interested in statistical properties of the mv solution such as its mean or its variance, the moments provide the perfect information, at least for sufficiently large $d$. However, if one is rather interested in properties of the graph of the entropy solution, a post processing step is required.

\textbf{An inverse problem} Recovering the graph of the solution $\{(t,x,y(t,x)): t\in \T, x\in \X\} \subset \T\times\X\times\Y$ from the moments of the measure $\nu=\lambda_\T\lambda_\X\delta_{y(t,x)}$ is an inverse problem whose detailed study is out of the scope of this paper, see e.g. \cite{claeys2015} in the context of controlled ODEs.  However, we briefly outline here one possible strategy with a formal justification. It turns out that it works surprising well in all our examples of the Burgers equation with or without shock.

Let $\genvar = (t,x,y)$ and $\mom_\alpha = \int_{\T\times\X\times\Y} \genvar^\alpha\d\nu$ denote the vector of moments of $\nu$. For any polynomial $p\in\rx_d$ with vector of coefficients $\p$ in the monomial basis, it holds
\[
\p^\top M_d(\mom) \p = \int p^2\d\nu.
\]
Consequently, if $\p$ is in the kernel of $M_d(\mom)$, we have that $$\int_{\T\times\X\times\Y} p^2\d\nu = 0.$$ In other words, the support of the measure is contained in the zero level set of every polynomial (whose vector of coefficients is) in the kernel of the moment matrix. 
However, this inclusion can be strict in some cases. Therefore we propose to also consider polynomials corresponding to small eigenvalues. Let us explain this now:

Since the moment matrix is positive semidefinite, it has a spectral decomposition
\begin{equation}\label{eig}
M_d(\mom) = P E P^\top
\end{equation}
where $P$ is an orthonormal matrix whose columns are denoted $\p_i$, $i=1,2,\ldots$ and satisfy $\p^\top_i \p_i = 1$ and $\p^\top_i \p_j = 0$ if $i\neq j$, and $E$ is a diagonal matrix whose diagonal entries are eigenvalues $e_{i+1} \geq e_i \geq 0$ of the moment matrix. Each column $\p_i$ is the vector of coefficients in the monomial basis of a polynomial $p_i \in\rx$, so that
\[
\p^\top_i M_d(\mom) \p_i = \int p^2_i\d\nu = e_i.
\]

The following result shows that the measure is concentrated on a sublevel set of an SOS polynomial constructed from the spectral decomposition of the moment matrix. 
\begin{lemma}[Concentration inequality]\label{concentration}
Let $r \in \N$ and $\beta>0$. Define
\[
\gamma = \frac{\sum_{i=1}^r e_i}{\beta}
\]
and  
\begin{equation}\label{sos}
\quad p_{\mathrm{sos}} = \sum_{i=1}^r p^2_i.
\end{equation}
Then
\[
\nu(\{\genvar : p_{\mathrm{sos}}(\genvar) \leq \gamma\}) \geq 1-\beta.
\]
\end{lemma}
The proof of Lemma \ref{concentration} follows readily from the inequality
\[
\nu(\{\genvar : p_{\mathrm{sos}}(\genvar) > \gamma\}) \leq \frac{\int p_{\mathrm{sos}} \d\nu}{\gamma} = \beta
\]
which holds since $\nu$ is a probability measure and $p_{\mathrm{sos}}$ is non-negative.
Lemma \ref{concentration} justifies the following algorithm, which extracts from a grid the values at which the polynomial $p_{\mathrm{sos}}$ is small:
\begin{description}
\item[Input] moment matrix $M_d(\mom)$ of measure $\nu =\lambda_\T\lambda_\X\delta_{y(t,x)}$, small real $\epsilon>0$, grid points \\ $(t_i,x_j,y_k)_{i,j,k=1,\ldots,N} \subset \T\times\X\times\Y$;
\item[Step 1] Compute spectral decomposition (\ref{eig}) of $M_d(\mom)$ and construct SOS polynomial $p_{\mathrm{sos}}$ in (\ref{sos}) with the largest number of terms $r$ such that $\sum_{i=1}^r e_i < \epsilon$;
\item[Step 2] For each $i,j,k=1,\ldots,N$, evaluate $p_{i,j,k}:=p_{\mathrm{sos}}(t_i,x_j,y_k)$;
\item[Step 3] For each $i,j=1,\ldots,N$, let $y_{i,j}:=y_{k^*}$ where $k^*:=\mathrm{arg}\:\min_k p_{i,j,k}$;
\item[Output] Approximation $(y_{i,j})_{i,j=1,\ldots,N} \subset \Y$ of $y(t,x)$ at grid points $(t_i,x_j)_{i,j=1,\ldots,N} \subset \T\times\X$.
\end{description}
The computational burden is modest: an eigenstructure decomposition at Step 1, and grid point evaluations of polynomial $p_{\mathrm{sos}}$ at Step 2.

\section{The Riemann problem for the Burgers equation}
\label{3.riemann}
For a numerical illustration, we consider the classical Riemann problem (see e.g., \cite{evans2010pde}) for a Burgers equation. In particular, we choose the flux 
\begin{equation*}
f(y) = \frac{1}{4}y^2.
\end{equation*}
The Riemann problem to this conservation law is a Cauchy problem with the following initial condition, piecewise constant with one point of discontinuity:
\begin{equation*}
y_0(x) = \left\{
\begin{array}{cl}
l &\mbox{if } x< 0,\\
r &\mbox{if } x> 0,
\end{array}
\right.
\end{equation*}
where $l,r\in\mathbb{R}$. The solution to the Riemann problem depends strongly on the values of $l$ and $r$. In particular:
\begin{enumerate}
\item If $l>r$, the shock at the initial condition spreads along the characteristics.
\item If $l<r$, the solution is not necessarily unique. The entropy condition allows to select the right solution, which is known as a rarefaction wave.   
\end{enumerate}
Both cases are interesting from a numerical point of view for their own reasons. In general, the first case is difficult to tackle because of the discontinuity. In general, numerical schemes based on discretization tend to smoothen out the shock. Indeed, recovering numerically the exact point of discontinuitiy is a challenge for these schemes. 

In the second case the solution is continuous, but not necessarily unique. For the Burgers equation, it has been shown that one single entropy condition is sufficient to guarantee uniqueness of the solution \cite{de2004minimal}. To the best of our knowledge, there is no similar result for the uniqueness of entropy mv solutions for Burgers equation with concentrated initial data, except for classical solutions \cite{demoulini2012weak}.

We present numerical results for both cases. We are going to consider $l,r\in\{0,1\}$. Following the discussion yielding (\ref{Y}), we can assume that the solution takes values only in $\Y=[0,1]$. The time-space-window on which we consider the solution is $\T=[0,1]$ and $\X =[L,R]=[-\tfrac{1}{2},\tfrac{1}{2}]$. 

Further note that, from the initial condition, we can derive that $$y(t,L)=l,\quad \forall t\in \T.$$ Moreover, due to positivity of $y$, the solution on $\T\times\X$ does not depend on the initial condition for $x>\tfrac{1}{2}$. 

\textbf{Remark on the significance of the numerical results upfront} We need to emphasize that these experiments are by no means conclusive. Our implementation is based on the Matlab interface Gloptipoly3 \cite{henrion2009gloptipoly} and the SDP solver of MOSEK \cite{mosek2012}.
The purpose of the numerical examples is to show that our framework actually works in practice and with a proper implementation might actually provide an alternative to schemes based on discretization.

\subsection{Shock waves}
Let $l=1$ and $r=0$. As it has been noticed before, with such an initial condition the solution is discontinuous, for all $t>0$. The unique analytical solution corresponding to this initial condition is
\begin{equation}
y^\ast(t,x) = \left\{
\begin{array}{cl}
1 & x> \frac{t}{4},\\
0 & x< \frac{t}{4}.
\end{array}
\right.
\end{equation}


As an objective function, we choose the default implemented in Gloptipoly, which minimizes the trace of the moment matrix.
Since the trace is the convex envelope of the rank on the set of matrices with norm less than one,
this is likely to cause early convergence of the moment-SOS hierarchy: low rank solutions correspond to measures supported on sets of zero Lebesgue measure. As in this case the marginal of $\optmeas$ with respect to $y$ is supported on $\{0,1\}$ we expect this criterion to be appropriate to accelerate convergence. Indeed, for $d=6$ (i.e. moments of degree up to $12$) we end up with the following moments for $y$:
\[
(\mom_{0,0,k})_{k=0,1,\ldots} = (1.0000,  0.6250,   0.6250,  0.6250,  0.6250,\ldots )
\]
which correspond (up to numerical accuracy) exactly with the moments of the analytic solution.

\begin{figure}
\centering{
\includegraphics[scale=0.9]{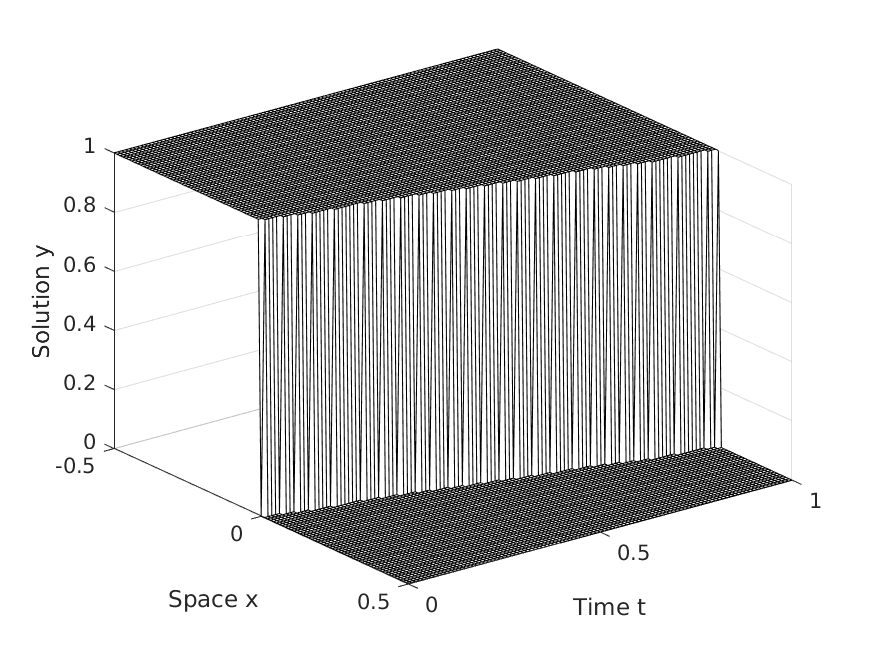}
}
\caption{Approximation of the solution $y(t,x)$ obtained with our GMP approach, in the case of a shock.}
\label{fig:shock}
\end{figure}

\textbf{Localizing the shock} 

In order to approximate the solution from our approximated moments we follow the path lined out in Section \ref{sec_extract}. 
Applying our algorithm with $\epsilon=10^{-6}$ yields a polynomial $p_{\mathrm{sos}}$ with $r=54$ terms in the approximate kernel of the moment matrix of size $84$, and to the approximated solution represented on Figure \ref{fig:shock}.

As already mentioned the computed moments can be used in order to approximate the location of the shock at some given time $t$. Here we will take $t=0.75$, consequently the shock is located at exactly $x=0.1875$. We used a standard Godunov scheme (we refer to \cite{randall1992numerical} for more details) to compute the solution up to this time. For space discretization, we took a mesh size of $0.0005$ and a consistent discretization in time such that the scheme stays stable. In Table \ref{tab:shock}, we display the obtained values from this approach on an interval around the shock. We can see the typical behaviour of shock smoothing.
In contrast, the values obtained by our GMP approach exactly represent the position of the shock.
\begin{table}
\begin{center}
\begin{tabular}{c|cccccccc}
$x$ 		& 0.1850 & 0.1855 & 0.1860 & 0.1865 & 0.1870 & 0.1875 & 0.1880 & 0.1885\\
\hline
Godunov	& 0.9999 & 0.9991 & 0.9936 & 0.9580 & 0.7647 & 0.2724 & 0.0123 & 0.0000\\
GMP		& 1.0000 & 1.0000 & 1.0000 & 1.0000 & 1.0000 & 0.0000 & 0.0000 & 0.0000
\end{tabular}
\caption{Approximation of $y(0.75,x)$ with Godunov and GMP.}
\label{tab:shock}
\end{center}
\end{table}

\subsection{Rarefaction waves}
Now let $l=0$ and $r=1$. As it has been noticed before, with such an initial condition, entropy conditions are crucial to select the right solution, i.e., the solution with a good physical meaning. The analytical entropy solution corresponding to this example is
\begin{equation}
y(t,x) = \left\{
\begin{array}{cl}
0 & x\leq 0,\\
\frac{2x}{t} & 0\leq x\leq\tfrac{t}{2},\\
1 & x\geq \frac{t}{2}.
\end{array}
\right.
\end{equation}

Numerically implementing all entropy pairs of Kruzkhov is possible (as seen in Section \ref{sec_mv-gmp}), but heavy. It is known that the entropy $\eta(y)=y^2$ provides all necessary information to make the entropy solution unique for Burgers equation \cite{de2004minimal}. Then, instead of using all Kruzkhov pairs, we propose the following family of entropies in this example: 
\begin{equation}\label{eq:pol entropy}
\eta_k(y)=y^{k},\qquad \forall k\in\mathbb{N}
\end{equation}
and the corresponding polynomial functions $q_k$. Note that $\eta$ is strictly convex on $\Y=[0,1]$. In particular, we do not have to split the measures in \eqref{linear-burgers2} into two measures, since there is no absolute value appearing in \eqref{eq:pol entropy}. It is neither necessary to introduce a lifting variable as was discussed in Section \ref{sec_mv-gmp}. Finally, we define the sum over all entropy constraints as an objective function to be maximized.

Solving the relaxation of order $d=6$ (i.e. moments of degree up to $12$), we obtain the following moments for the marginal on $y$:
\[
(\mom_{0,0,k})_{k=0,1,\ldots} = (1.0000,	0.3750,	0.3333,	0.3125,	0.3000,	0.2917,	0.2857,	0.2812,\ldots)
\]
which, again up to numerical accuracy, coincide with the moments of the actual analytic entropy solution. Applying the algorithm from Section \ref{sec_extract} with $\epsilon=10^{-6}$ yields a polynomial $p_{\mathrm{sos}}$ with $r=48$ terms in the approximate kernel of the moment matrix of size $84$, and the approximated solution represented on Figure \ref{fig:rarefaction}.
\begin{figure}
\centering{
\includegraphics[scale=0.9]{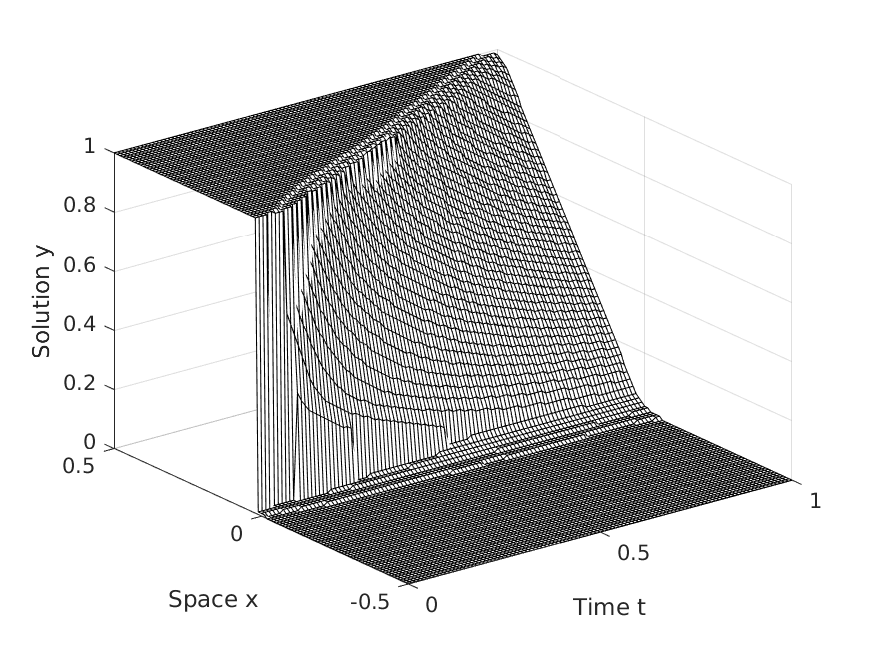}
}
\caption{Approximation of the solution $y(t,x)$ obtained with our GMP approach, in the case of a rarefaction wave.}
\label{fig:rarefaction}
\end{figure}

\section{Conclusion}

\label{5.conclusion}

In this paper, we have provided a new method to solve scalar polynomial hyperbolic partial differential equations. This method relies on the moment-SOS hierarchy surveyed in \cite{lasserre2009moments}. More precisely, we have proved that the truncated moments associated to the measure-valued solution formulation converge to the Dirac measure concentrated on entropy solution to the scalar polynomial solution. we believe that all the arguments of our paper extend to the case of a spatial variable $x$ of dimension greater than one.

The idea of solving linear problems on measures to solve nonlinear differential equation is not new. In the context of nonlinear ordinary differential equations (ODEs), the linear problems involved measures called {occupation measures}. Roughly speaking, occupation measures allow to measure the time spent by a graph of the trajectory of the ODE in a given subset of the state space. Provided that the nonlinearities considered are polynomial, one can transform the nonlinear ODE into a linear {moment problem}, in turn solved numerically with the moment-SOS hierarchy, see
\cite{lasserre2008nonlinear} and the survey \cite{pauwels2017}. Therefore, the current paper can be seen as an extension to (uncontrolled) PDEs of the results provided in \cite{lasserre2008nonlinear} for (controlled) ODEs. 

This opens many further research lines. For example:
\begin{itemize}
\item One of the most interesting aspect of the notion of very weak solution is the linear formulation on measures of nonlinear differential equations. Such formulations have been useful to solve many problems appearing in the ODE framework, such as optimal control \cite{lasserre2008nonlinear} or approximation of region of attraction \cite{henrion2014convex}. The challenge was to prove that the measure formulation was not a relaxation of the original nonlinear problem. For the hyperbolic conservation law studied in our paper, we have used entropy inequalities for that purpose. We are wondering whether it is possible to extend these techniques to the case of other nonlinear PDEs.
\item A class of other nonlinear PDEs could be parabolic ones. One of the interest of these equations is that they regularize the solution, whatever is the initial condition. Therefore, as it is done for ODEs in \cite{lasserre2008nonlinear}, it might be possible to define test functions depending on the solution to the parabolic equation and then define an occupation measure associated to the latter. This together with the relaxed control theory surveyed in \cite{fattorini1999infinite} might be instrumental to solve optimal control problem for nonlinear parabolic equations.
\item The Burgers equation is irreversible. Roughly speaking, given a terminal condition $y(T,x)$ with $T>0$, there exists a continuum of initial conditions yielding $y(T,x)$, see e.g., \cite{gosse2017filtered}. Such a continuum can be described with measures and, hence, our linear formulation might be useful to solve such inverse problems, extending to PDEs what was developed in \cite{henrion2014convex} for ODEs.
\end{itemize}

\textbf{Acknowledgement:} The authors would like to thank Matthieu Barreau for his help with the Godunov numerical scheme and Sylvain Ervedoza for all the interesting and encouraging discussions. This work also benefited from feedback from Yann Brenier, Bruno Despr\'es, Maxime Herda, Milan Korda, Ond\v rej Kreml and Josef M\'alek.

\appendix

\section{Proof of Theorem \ref{thm-concentration}}\vspace{0.5cm}

\label{proof-concentration}

The proof is divided into two steps. The first step consists in proving that $v$ can be replaced by $y(t,x)$ or $\langle \mu_{(t,x)},y\rangle$. The second step aims at proving the contraction inequality given in \eqref{contraction}. In each step, a special choice of test function is done in order to prove the result.\\

\textbf{$\bullet$ First step: Doubling variable} \\

Let us consider the entropy pair given in \eqref{kruzkhov-entropy}. For all $(s,z)\in\R_+\times \R$, we choose $v=y(s,z)$, where $y$ is an entropy solution to \eqref{eq:burgers:compact}:
\begin{equation}
\begin{split}
\int_0^T\int_{\R} &\left( \frac{\partial \psi}{\partial t} \langle \mu_{(t,x)},|y-y(s,z)|\rangle + \frac{\partial \psi}{\partial x} \langle \mu_{(t,x)},\sign(y-y(s,z))(f(y)-f(y(s,z))\rangle\right) dx\:dt \\
&+ \int_{\R} \psi(0,x)\langle \measO,|y-y(s,z)|\rangle dx - \int_{\R} \psi(T,x) \langle \mu_{(t,x)}, |y - y(s,z)|\rangle dx.
\end{split}
\end{equation}
Similarly, for all $(s,z)\in\R_+\times\R$, we set $v=\langle \mu_{(t,x)},y\rangle$ in  \eqref{entropy-sol} and use the fact that $\mu$ is a probability measure:
\begin{equation}
\begin{split}
\int_0^T\int_{\R} &\left(\frac{\partial \psi}{\partial s} \langle \mu_{(t,x)},|y-y(s,z)|\rangle + \frac{\partial \psi}{\partial z} \langle \mu_{(t,x)},\sign(y-y(s,z))(f(y)-y(s,z))\rangle\right) ds\,dz \\
&+ \int_{\R} \psi(0,z)\langle \mu_{(t,x)},|y-y_0(z)|\rangle dz  \geq 0.
\end{split}
\end{equation}

Let us choose $\tilde{\psi}:=\tilde{\psi}(t,x,s,z)$. Thanks to the two latter inequalities, one has
\begin{equation}
\label{Kruzkhov}
\begin{split}
&\int_{\mathbb{R}_+}\int_\mathbb{R}\int_{\mathbb{R}_+}\int_{\mathbb{R}} \left(\left(\frac{\partial \psi}{\partial s}+\frac{\partial \psi}{\partial t}\right)\langle \mu_{(t,x)},|y-y(s,z)|\rangle\right. \\
&\left.+ \left(\frac{\partial \psi}{\partial z}+\frac{\partial \psi}{\partial x}\right)\langle \mu_{(t,x)},\sign(y-y(s,z))(f(y)-f(y(s,z)))\rangle\right) ds\, dz\, dx \,dt\\
&+ \int_{\mathbb{R}}\int_\mathbb{R} \psi(t,x,0,z)\langle \mu_{(t,x)},|y-y_0(z)|\rangle dz \,dx \\
&+ \int_{\mathbb{R}}\int_\mathbb{R} \psi(0,x,s,z)\langle \measO,|y-y(s,z)|\rangle dz\, dx \geq 0.
\end{split}
\end{equation}

Let $\rho_\varepsilon \in \C^1(\R)$ satisfy
\begin{equation}
\int_{\R}\rho_\varepsilon(s) ds = 1. 
\end{equation}
For all $\varphi\in \C^1_c(\R_+\times\R)$, one defines $\tilde{\psi}$ as follows
\begin{equation}
\tilde{\psi}(t,x,s,z) = \frac{1}{\varepsilon^2}\rho_\varepsilon\left(\frac{t-s}{2\varepsilon}\right)\rho_\varepsilon\left(\frac{x-z}{2\varepsilon}\right)\varphi\left(\frac{t+s}{2},\frac{x+z}{2}\right).
\end{equation}
Therefore
\begin{equation}
\label{psit}
\left(\frac{\partial\tilde{\psi}}{\partial t}\tilde{\psi}+\frac{\partial \tilde{\psi}}{\partial s}\right)= \frac{1}{\varepsilon^2}\frac{\partial\varphi}{\partial t}\left(\frac{t+s}{2},\frac{x+z}{2}\right)\rho_\varepsilon\left(\frac{x-z}{2\varepsilon}\right)\rho_\varepsilon\left(\frac{t-s}{2\varepsilon}\right)
\end{equation}
and
\begin{equation}
\label{psix}
\left(\frac{\partial\tilde{\psi}}{\partial x}+\frac{\partial\tilde{\psi}}{\partial z}\right)= \frac{1}{\varepsilon^2}\frac{\partial\varphi}{\partial x}\left(\frac{t+s}{2},\frac{x+z}{2}\right))\rho_\varepsilon\left(\frac{x-z}{2\varepsilon}\right)\rho_\varepsilon\left(\frac{t-s}{2\varepsilon}\right).
\end{equation}
We aim at proving that
\begin{equation}
\rho_\varepsilon\left(\frac{t-s}{\varepsilon}\right)\rightarrow \delta(t-s)\text{ as }\varepsilon\rightarrow 0,
\end{equation}
and
\begin{equation}
\rho_\varepsilon\left(\frac{x-z}{\varepsilon}\right)\rightarrow \delta(x-z)\text{ as }\varepsilon\rightarrow 0,
\end{equation}
so that we will have, thanks to \eqref{psit} and \eqref{psix}
\begin{equation}
\begin{split}
\int_{\R_+}\int_{\R} &\left(\frac{\partial \varphi}{\partial t}\langle \mu_{(t,x)},|y-y(t,x)|\rangle + \frac{\partial \varphi}{\partial x}\langle \mu_{(t,x)},\sign(y-y(t,x))(f(y)-f(y(t,x))\rangle\right) dx\,dt\\
&+ \int_{\R} \varphi(0,x)\langle \mu_{(t,x)},|y-y_0(x)|\rangle dx  \geq 0.
\end{split}
\end{equation}
Noticing that, for a fixed $t>0$,
\begin{equation}
2\int_{\R} \frac{1}{2\varepsilon} \rho\left(\frac{t-s}{\varepsilon}\right) ds = 2,
\end{equation}
then, up to a change of variable, one has, for any continuous function $\phi:=\phi(t,s)$
\begin{equation}
\int_{\R} \int_{\R}\frac{1}{\varepsilon} \rho\left(\frac{t-s}{\varepsilon}\right)\phi(t,s) dt \,ds= 2\int_{\R}\phi(t,t) dt.
\end{equation}
Since $\phi$ is continuous and lies in a compact set, then it is uniformly continuous. Therefore, one has
\begin{equation}
|t-s|\leq 2\varepsilon \Rightarrow |\phi(t,s)-\phi(t,t)|\leq M_\varepsilon.
\end{equation} 
Therefore, for any positive value $k$, it follows that
\begin{equation}
\begin{split}
\left| \int_{\R_+}\int_{\R_+} \frac{1}{\varepsilon}\rho\left(\frac{t-s}{2\varepsilon}\right) \phi(t,s)-\phi(t,t)dt \,ds\right| \leq &  M_\varepsilon \int_{\R}\int_{|t-s|\leq 2\varepsilon,|s|\leq k,|t|\leq k} \frac{1}{\varepsilon}\rho\left(\frac{t-s}{2\varepsilon}\right) dt \,ds\\
\leq & 4 M_{\varepsilon}k.
\end{split}
\end{equation}
Note that
\begin{equation}
M_\varepsilon\rightarrow 0\text{ as }\varepsilon\rightarrow 0.
\end{equation}
Hence,
\begin{equation}
\left| \int_{\R_+}\int_{\R_+} \frac{1}{\varepsilon}\rho\left(\frac{t-s}{2\varepsilon}\right) (\phi(t,s)-\phi(t,t)) dt \,ds\right| \rightarrow 0\text{ as }\varepsilon\rightarrow 0.
\end{equation}
Similarly, one can prove that, for any continuous function $\tilde{\phi}:=\tilde{\phi}(x,z)$
\begin{equation}
\left| \int_{\R}\int_{\R} \frac{1}{\varepsilon}\rho\left(\frac{x-z}{2\varepsilon}\right) (\tilde{\phi}(x,z)-\tilde{\phi}(x,x))dt \,ds\right| \rightarrow 0\text{ as }\varepsilon\rightarrow 0.
\end{equation}
Finally, using \eqref{psit} and \eqref{psix}, the equation \eqref{Kruzkhov} converges to
\begin{equation}
\label{firststep}
\begin{split}
\int_{\R_+}\int_{\R}& \left(\frac{\partial\varphi}{\partial t} \langle \mu_{(t,x)},|y - y(t,x)|\rangle + \frac{\partial \varphi}{\partial x}\langle \mu_{(t,x)},\sign(y -y(t,x))(f(y)-f(y(t,x)))\right) dx \,dt \\
&+ \int_{\R} \varphi(0,x) \langle \measO,|y-y_0(x)|\rangle dx\geq 0,
\end{split}
\end{equation}
as $\varepsilon$ goes to $0$. \\

\textbf{$\bullet$ Second step: Contraction inequality}\\

Given two positive values $r\in\R$ and $T\in\R_+$, let us choose $\varphi$ as follows:
\begin{equation}
\varphi(t,x) = \theta_{\varepsilon}(t)\Delta_{\varepsilon}(t,x),\quad (t,x)\in\R_+\times \R
\end{equation}
where
\begin{equation}
\theta_{\varepsilon}(t):=\left\{
\begin{split}
& 1  & t \leq T,\\
& \frac{T-t}{\varepsilon} &T\leq t \leq T+\varepsilon,\\
& 0  &t\geq T+\varepsilon,
\end{split}
\right.
\end{equation}
and
\begin{equation}
\Delta_{\varepsilon}(t,x) : = \left\{
\begin{split}
& 1 &|x|\leq r + C(T-t),\\
& \frac{r+C(T-t)-|x|}{\varepsilon} & r+C(T-t)\leq |x|\leq r+C(T-t)+\varepsilon,\\
&0  &|x|\geq r+C(T-t)+\varepsilon
\end{split}
\right.
\end{equation}
where $C$ denotes the Lipschitz constant of the function $f$.
Differentating $\theta_\varepsilon$ with respect to $t$ yields
\begin{equation}
\frac{\partial\theta_{\varepsilon}(t)}{\partial t} :=\left\{
\begin{split}
& -\frac{1}{\varepsilon} & T\leq t \leq T+\varepsilon,\\
& 0 
\end{split}
\right.
\end{equation}
Differentiating $\Delta_\varepsilon$ with respect to $t$ yields
\begin{equation}
\frac{\partial \Delta_{\varepsilon}}{\partial t}:= \left\{
\begin{split}
& -\frac{C}{\varepsilon} & r+C(T-t)\leq |x|\leq r+C(T-t)+\varepsilon,\\
& 0 & |x|\geq r+C(T-t)+\varepsilon,
\end{split}
\right.
\end{equation}
Finally, differentiating $\Delta_\varepsilon$ with respect to $x$ yields
\begin{equation}
\frac{\partial \Delta_{\varepsilon}}{\partial x}:= \left\{
\begin{split}
& -\frac{\sign(x)}{\varepsilon} & r+C(T-t)\leq |x|\leq r+C(T-t)+\varepsilon,\\
& 0 & |x|\geq r+C(T-t)+\varepsilon.
\end{split}
\right.
\end{equation}
Hence, using these functions in \eqref{firststep}, one has
\begin{equation}
\label{contraction1}
\begin{split}
&-\frac{1}{\varepsilon}\int_{T}^{T+\varepsilon} \int_{\R} \Delta_{\varepsilon}(t,x)\langle \mu_{(t,x)},|y - y(t,x)|\rangle dt \,dx \\
&-\frac{L}{\varepsilon}\int_0^T \int_{r+C(T-t)\leq |x|\leq r+C(T-t)+\varepsilon} \langle \mu_{(t,x)},|y- y(t,x)|\rangle dt \,dx \\
&- \frac{1}{\varepsilon} \int_{\mathbb{R}_+}\int_{r+C(T-t)\leq |x|\leq r+C(T-t)+\varepsilon} \sign(x)\langle \mu_{(t,x)},\sign(y-y(t,x))(f(y) - f(y(t,x)))\rangle dt \,dx \\
&+ \int_{\R} \varphi(0,x) \langle \measO,|y - y_0(x)|\rangle dx\geq 0.
\end{split}
\end{equation}
Note that one has
\begin{equation}
\sign(x)\langle \mu_{(t,x)},\sign(y-y(t,x))(f(y) - f(y(t,x)))\rangle \leq C\langle \mu_{(t,x)},|y - y(t,x)|\rangle.
\end{equation}
Therefore, \eqref{contraction1} becomes
\begin{equation*}
\frac{1}{\varepsilon}\int_{T}^{T+\varepsilon} \int_{\R} \Delta_{\varepsilon}(t,x) \langle \mu_{(t,x)},|y - y(t,x)|\rangle dt \,dx \leq \int_{\R} \varphi(0,x) \langle \measO,|y - y_0(x)|\rangle dx.
\end{equation*}
From this latter equation, one can conclude the proof when $\varepsilon$ goes to $0$. Indeed, the left hand side of the inequality can be bounded as follows:
\begin{equation}
\begin{split}
\frac{1}{\varepsilon}&\left|\int_{T}^{T+\varepsilon} \int_{\R} \Delta_{\varepsilon}(t,x) \langle \mu_{(t,x)},|y - y(t,x)|\rangle dt \,dx\right| \\
\leq &\frac{1}{\varepsilon}\left|\int_{T}^{T+\varepsilon} \int_{|x|\leq r +\varepsilon} \frac{r+C(T-t)-|x|}{\varepsilon} \langle \mu_{(t,x)},|y - y(t,x)|\rangle dt\,dx\right|\\ 
\leq & \frac{1}{\varepsilon}\int_{T}^{T+\varepsilon} \int_{|x|\leq r+\varepsilon} \left|\frac{r+C(T-t)-|x|}{\varepsilon}\right|\left|\langle \mu_{(t,x)},|y - y(t,x)|\rangle\right| dt \,dx.
\end{split}
\end{equation}
Noticing that, for any $t\in [T,T+\varepsilon]$ and any $|x|\leq r+\varepsilon$, one has
\begin{equation}
\left|\frac{r+C(T-t)-|x|}{\varepsilon}\right|\leq 1,
\end{equation}
then one has
\begin{equation}
\frac{1}{\varepsilon}\int_{T}^{T+\varepsilon} \int_{\R} \Delta_{\varepsilon}(t,x) \langle \mu_{(t,x)},|y - y(t,x)|\rangle dt \,dx \leq \frac{1}{\varepsilon}\int_{T}^{T+\varepsilon} \int_{|x|\leq r+\varepsilon} \left|\langle \mu_{(t,x)},|y - y(,x)|\rangle\right| dt \,dx.
\end{equation}
Moreover, one has
\begin{equation}
\frac{1}{\varepsilon}\int_{T}^{T+\varepsilon} \int_{|x|\leq r+\varepsilon} \left|\langle \mu_{(t,x)},|y- y(,x)|\rangle\right| dt \,dx \rightarrow \int_{|x|\leq r} \langle \measT,|y-y(T,x)|\rangle dx\text{ as }\varepsilon\rightarrow 0.
\end{equation}
Similarly, one can prove that
\begin{equation}
\int_{\R} \varphi(0,x) \langle \measO,|y - y_0(x)|\rangle dx \rightarrow \int_{|x|\leq r+CT} \langle \measO,|y - y_0(x)|\rangle dx \text{ as }\varepsilon \rightarrow 0.
\end{equation}
Finally, it yields, for all positive values $T$ and $r$
\begin{equation}
\int_{|x|\leq r} \langle \measT,|y-y(T,x)|\rangle dx \leq \int_{|x|\leq r+CT} \langle \measO,|y - y_0(x)|\rangle dx.
\end{equation}
This concludes the proof of Theorem \ref{thm-concentration}.

\bibliographystyle{plain}

\end{document}